\documentclass[11pt, twoside, leqno]{amsart}  
\usepackage{lipsum}
\usepackage{amsfonts}
\usepackage{graphicx}
\graphicspath{ {figures/} }
\usepackage{epstopdf}
\usepackage{listings}
\usepackage{amsaddr}
\usepackage{float} 
\usepackage[utf8]{inputenc}
\usepackage{pgfplots}
\DeclareUnicodeCharacter{2212}{−}
\usepgfplotslibrary{groupplots,dateplot}
\usetikzlibrary{patterns,shapes.arrows}
\pgfplotsset{compat=newest}
\usepackage{calligra}
\usepackage{amsfonts,amsmath,amsthm,amssymb}
\usepackage{mathtools}
\usepackage{hyperref}
\usepackage{autonum}
\usepackage{hhline}
\usepackage{mathtools}
\usepackage{array}
\usepackage[normalem]{ulem}
\usepackage{diagbox}
\usepackage{tcolorbox}
\usepackage{mdframed}
\usepackage{multicol}
\usepackage{graphicx}
\usepackage{subcaption}
\usepackage{moreverb}
\usepackage{bbm}
\usepackage[margin=1.38in]{geometry}
\usepackage{todonotes}
\usepackage{scalerel,amssymb}
\allowdisplaybreaks
\usepackage{mathrsfs}  
\usepackage{lineno}
\usepackage{todonotes}
\usepackage{tikz}
\usepackage{pgfplots}
\usetikzlibrary{arrows.meta}
\usepackage[numbers,sort&compress]{natbib}
\definecolor{mygreen}{HTML}{43a047} 
\usepackage{subcaption}
\usepackage{doi}

\usepackage{ulem}

\newcommand{\Vanja}[1]{\textcolor{blue}{#1}}
\newcommand{\vanja}[1]{\footnote{\textcolor{blue}{#1}}}

\def\ulal{\underline{\alpha}}
\def\olal{\overline{\alpha}}

\def\olbe{\overline{\beta}}

\def\opsi{\overline{\psi}}

\newcommand{\Om}{\Omega}
\newcommand{\D}{\Delta}

\newcommand{\tildePv}[1]{\tilde{\boldsymbol P}_h #1}
\newcommand{\tildePp}[1]{\tilde{P}_h  #1}
\renewcommand{\div}{\nabla \cdot}
\newcommand{\Ih}{\boldsymbol{I}_h}
\newcommand{\bv}{\boldsymbol{v}}

\newcommand{\bw}{\boldsymbol{w}}
\newcommand{\bwh}{\boldsymbol{w}_h}
\newcommand{\bvh}{\boldsymbol{v}_h}
\newcommand{\obvh}{\overline{\boldsymbol{v}}_h}
\newcommand{\obvht}{\overline{\boldsymbol{v}}_{ht}}

\newcommand{\bvt}{\boldsymbol{v}_t}
\newcommand{\bvht}{\boldsymbol{v}_{ht}}

\newcommand{\ph}{\psi_h}
\newcommand{\pt}{\psi_t}
\newcommand{\pht}{\psi_{ht}}

\newcommand{\pth}{\psi_{ht}}
\newcommand{\ptt}{\psi_{tt}}
\newcommand{\phtt}{\psi_{htt}}
\newcommand{\ptth}{\psi_{htt}}
\newcommand{\ddt}{\frac{\textup{d}}{\textup{d}t}}
\newcommand{\dt}{\, \textup{d} t}
\newcommand{\ds}{\, \textup{d} s }

\newcommand{\intT}{\int_0^t}

\newcommand{\nLtwo}[1]{\|#1\|_{L^2}}

\newcommand{\nLinf}[1]{\|#1\|_{L^\infty}}

\newcommand{\nLoneLtwot}[1]{\|#1\|_{L^1 (0,t;L^2)}}
\newcommand{\nLtwoLtwot}[1]{\|#1\|_{L^2 (0,t;L^2)}}

\newcommand{\nLtwoLinf}[1]{\|#1\|_{L^2 (L^\infty)}}

\newcommand{\nLinfLinf}[1]{\|#1\|_{L^\infty (L^\infty)}}

\newcommand{\prodLtwo}[2]{(#1, #2)_{L^2}}
\newcommand{\R}{\mathbb{R}} 
 
\newcommand{\Hdiv}{\boldsymbol{H}(\textup{div}; \Omega)}
\newcommand{\spaceV}{\boldsymbol{V}_{h}}
\newcommand{\spaceS}{\Psi_{h}}

\newcommand{\Ltwo}{L^2(\Omega)}

\newtheorem{theorem}{Theorem}
\newtheorem{lemma}{Lemma}
\newtheorem{proposition}{Proposition}
\newtheorem{assumption}{Assumption}

\newtheorem{remark}{Remark}
\numberwithin{lemma}{section}
\numberwithin{proposition}{section}
\numberwithin{theorem}{section}
\numberwithin{equation}{section}
\numberwithin{corollary}{section}
\makeatletter
\newcommand{\leqnomode}{\tagsleft@true}
\newcommand{\reqnomode}{\tagsleft@false}
\makeatother

\definecolor{RUred}{rgb}{0.8902,0.4471,0.1333}
\definecolor{tumblau}{rgb}{0,0.,1.}
\definecolor{tumgruen}{rgb}{0.,0.5,0.}
\definecolor{tumelfenbein}{rgb}{0.8549,0.8431,0.7961}
\definecolor{magentao}{rgb}{0.56,0.0,0.56}
\definecolor{cyano}{rgb}{0.0,0.56,0.56}
\definecolor{red}{rgb}{0.8902,0.3471,0.1333}
\definecolor{redd}{rgb}{0.97502,0.1071,0.05333}	
\definecolor{blued}{rgb}{0,0.2,0.349019608}
\definecolor{blue}{rgb}{0.301961,0.443137,0.5490196}
\definecolor{bluel}{rgb}{0.454902,0.564706,0.6570588}
\definecolor{tumdarkblue}{rgb}{0,0.2,0.349019608}
\definecolor{tumdarkbluel}{rgb}{0.152941,0.321569,0.45098} 
\definecolor{tumdarkbluen}{rgb}{0.301961,0.443137,0.5490196} 
\definecolor{tumdarkbluem}{rgb}{0.454902,0.564706,0.6570588} 
\definecolor{tumblaus}{rgb}{0.603922, 0.760784, 0.901961} 
\definecolor{tumblaum}{rgb}{0.403922, 0.639216, 0.847059} 
\definecolor{light-gray}{gray}{0.02}
\definecolor{notlight-gray}{gray}{0.35}
\definecolor{myblue}{rgb}{93, 188, 210 }
\definecolor{RUred}{rgb}{0.745,0.192,0.102}
\definecolor{RUblack}{rgb}{0,0,0}
\definecolor{RUwhite}{rgb}{0.98,0.98,0.98}
\definecolor{amber}{rgb}{1.0, 0.75, 0.0}

\definecolor{grey}{rgb}{0.5,0.5,0.5}
 
\newtheorem{assumptionalt}{Assumption}[assumption]

\newenvironment{assumptionp}[1]{
	\renewcommand\theassumptionalt{#1}
	\assumptionalt
}{\endassumptionalt}

\newcommand{\BK}{\mathcal{B}_{\textup{K}}}

\title[Mixed approximation of nonlinear acoustic equations]{Mixed approximation of nonlinear acoustic equations:
	Well-posedness and \emph{a priori} error analysis}


\author[M. Meliani and V. Nikoli\'{c}]{\small Mostafa Meliani and Vanja Nikoli\'{c}}
\address{ 
	Department of Mathematics \\ 
	Radboud University   \\ 
	Heyendaalseweg 135,
	6525 AJ Nijmegen, The Netherlands}
\email{mostafa.meliani@ru.nl} 
\email{vanja.nikolic@ru.nl}
\begin{document}
\vspace*{8mm}
\begin{abstract}
	Accurate simulation of nonlinear acoustic waves is essential for the continued development of a wide range of (high-intensity) focused ultrasound applications. This article explores mixed finite element formulations of classical strongly damped quasilinear models of ultrasonic wave propagation; the Kuznetsov and Westervelt equations. Such formulations allow simultaneous retrieval of the \emph{acoustic particle velocity} and either the \emph{pressure} or \emph{acoustic velocity potential}, thus characterizing the entire ultrasonic field at once. Using non-standard energy analysis and a fixed-point technique, we establish sufficient conditions for the well-posedness, stability, and optimal \textit{a priori} errors in the energy norm for the semi-discrete equations. For the Westervelt equation, we also determine the conditions under which the error bounds can be made uniform with respect to the involved strong dissipation parameter. A byproduct of this analysis is the convergence rate for the inviscid (undamped) Westervelt equation in mixed form. Additionally, we discuss convergence in the $L^q(\Omega)$ norm for the involved scalar quantities, where $q$ depends on the spatial dimension. Finally, computer experiments for the Raviart--Thomas (RT) and Brezzi--Douglas--Marini (BDM) elements are performed to confirm the theoretical findings.
\end{abstract}
\vspace*{-7mm}
\maketitle           

\section{Introduction}
\indent In various ultrasound applications, characterizing the whole acoustic field accurately (that is, the scalar  pressure or potential fields as well as the velocity) is of particular importance as it allows precise computing of different quantities of interest, such as the acoustic intensity. With this motivation in mind, the purpose of this article is to investigate mixed-finite-element approximations for a family of strongly damped acoustic wave models that describe ultrasound propagation, the most general of which is the Kuznetsov equation.

\indent The Kuznetsov equation~\cite{kuznetsov1971equations} is a popular model of nonlinear sound propagation through fluids that accounts for thermoviscous dissipation and general nonlinearities of quadratic type. In the mathematical literature, it is commonly stated as
\begin{equation}\label{eq:continuous_problem}
	\left. \begin{aligned}
		&(1+2k \psi_t)\psi_{tt}-c^2 \D \psi-b \D \psi_t+2\sigma \nabla \psi \cdot \nabla \psi_t=0,
	\end{aligned}\right.
\end{equation}
where $\psi$ is the acoustic velocity potential. The constant $c$ denotes the speed of sound and $b$ the sound diffusivity.  The constant $k \in \R$ is a function of the coefficient of nonlinearity of the medium $\beta_a$ and the propagation speed. Although $\sigma = 1$ for the Kuznetsov equation, we generalize it to $\sigma \in \R$ in this work, so that a particular choice of $(k, \sigma)$ will allow us to retrieve the Westervelt equation~\cite{westervelt1963parametric} ($\sigma=0$) and the linear strongly damped equation ($k=\sigma=0$) as well. We refer the reader to the books~\cite{kaltenbacher2007numerical, hamilton1998nonlinear} for more physical background on the Kuznetsov equation and other models of nonlinear acoustics. \\ 
\indent Since $\bv=\nabla \psi$ is the acoustic particle velocity, we can rewrite \eqref{eq:continuous_problem} in a potential-velocity form, adding a general source term $f$, and couple it with homogeneous Dirichlet boundary conditions as well as initial conditions, to arrive at
\begin{equation}\label{eq:continuous_mixed_problem}
	\left \{ \begin{aligned}
		&(1+2k \psi_t)\psi_{tt}-c^2 \nabla \cdot \bv -b \, \nabla \cdot \bvt+2\sigma  \boldsymbol{v} \cdot \boldsymbol{v}_t=f \ \text{in } \Omega \times (0,T), \\[2mm]
		& \boldsymbol{v}= \nabla \psi,\\[2mm]
		& \psi|_{\partial\Omega} = 0 , \quad \psi(0)=\psi_0,\quad \pt(0)=\psi_1.
	\end{aligned}\right.
\end{equation} 
Approximating $(\psi(t), \bv(t))$ amounts to solving \eqref{eq:continuous_mixed_problem} in a finite-dimensional subspace $\spaceS\times \spaceV$ of $ \Ltwo\times \Hdiv$ for $t \in [0,T]$, where $h>0$ is a spatial discretization parameter. We aim to provide convergence rates for some of the most popular approximation spaces for $\Hdiv$; see e.g., \cite{raviart1977mixed,nedelec1980mixed,nedelec1986new,brezzi1985two,brezzi1987efficient,brezzi1987mixed,chen1989prismatic}. 

In mixed finite element formulations, the \textit{acoustic particle velocity} and either \textit{acoustic velocity potential} or \emph{acoustic pressure} (corresponding in elastodynamics to \textit{stress} and \textit{displacement}/\textit{velocity}, respectively) are approximated at the same time, resulting in a higher order of accuracy in the approximation of velocity~\cite{cowsat1990priori}. This property is particularly useful
for applications where the gradient of the acoustic field is of importance, such as enforcing absorbing conditions~\cite{shevchenko2015absorbing} or gradient-based shape optimization of focused ultrasound devices~\cite{kaltenbacher2016shape,meliani2021shape}.

Moreover, the mixed finite element approach allows us to characterize the whole acoustic field at once as well as compute acoustic intensity accurately. This is relevant for, among others, ultrasound-induced heating of biological tissue as a result of acoustic absorption. In such models, the acoustic energy flux (acoustic intensity) acts as a source term for the temperature equation; see, e.g.,~\cite{shevchenko2012multi}.

\indent The error estimates of spatial approximations using continuous and discontinuous Galerkin elements are available for the damped Westervelt equation ($b>0$, $\sigma=0$); see~\cite{nikolic2019priori, antonietti2020high}.  A discontinuous Galerkin coupling for nonlinear elasto-acoustics based on the damped Kuznetsov equation ($b>0$, $\sigma \neq0$) has been analyzed in \cite{muhr2021discontinuous}. However, to the best of our knowledge, this is the first work on the rigorous analysis of mixed formulations for classical models of nonlinear acoustics. We note that mixed finite element methods have been extensively studied in the context of linear wave equations; see~\cite{peralta2022mixed,egger2020mass,geveci1988application,cowsat1990priori,cowsar1996priori,jenkins2002priori,kirby2015symplectic, makridakis1992mixed} and the references contained therein. For the \emph{a priori} analysis of a strongly damped linear wave equation in mixed form, we refer to \cite{pani2001mixed}. Some work exists also on convergence of the approximations of nonlinear wave equations; see~\cite{chen2001improved}.
However, due to specific nonlinear terms $(1+2k\psi_t)\psi_{tt}$ and $\nabla \psi \cdot \nabla \psi_t$ present in the Kuznetsov equation, our analysis differs from the aforementioned references and is of particular relevance for simulation of high-intensity focused ultrasound. The low regularity of the mixed FEM space ($\psi$ approximated in $\Ltwo \not\subset L^\infty(\Omega)$) will require the use of inverse estimates to ensure the positiveness of $(1+2k\psi_{ht})$ and thus the non-degeneracy of the semi-discrete model, in combination with suitable smallness assumptions on the exact solution and the discretization parameter.\\
\indent The damping parameter $b$ is relatively small in practice and can become negligible in certain media; see~\cite[Ch.\ 5]{kaltenbacher2007numerical}. Motivated by this, we additionally provide a uniform-in-$b$ error analysis of the mixed approximation of the Westervelt equation $(\sigma=0)$. In particular, we establish sufficient conditions under which  the hidden constant in the derived bounds does not degenerate as $b \rightarrow 0^+$.  Nonlinear acoustic equations are notoriously harder to treat in such a uniform manner. In the continuous setting, they require the use of energy functionals of higher order compared to the non-uniform analysis; see~\cite{kaltenbacher2022parabolic,kaltenbacher2009global}. In a finite-dimensional setting, one cannot expect to have access to such energies due to, in general, low global spatial regularity of the numerical solution. Instead, we will exploit the time-differentiated semi-discrete version of the first equation in \eqref{eq:continuous_mixed_problem}. Concerning related results, to our knowledge, this is the first result dealing with the uniform mixed approximation of the Westervelt equation. The finite element estimates for the Westervelt equation in case $b=0$  and in standard (non-mixed) form follow as a particular case of the results in~\cite{hochbruck2021error,maier2020error}. \\
\indent Our theoretical approach relies on non-standard energy estimates for a linearized and non-degenerate problem presented in Section~\ref{sec:lin}, followed by a fixed-point argument in Section~\ref{sec:nonlin_kuz}, 
where we show that the previously established estimates hold provided a combination of smallness of discretization step and of exact solution norm is satisfied. 
In Section~\ref{sec:inv_West}, we present the stability and error results relating to the inviscid Westervelt equation. We point out that one of the crucial elements of the analysis is the choice of the approximate initial data; in particular, they should be chosen as mixed projections of the exact data; see \eqref{eq:init_propL} below for details. The main theoretical results are contained in Theorems~\ref{Thm:Kuzn} and~\ref{Thm:Wes}. In Section~\ref{sec:numexp}, we provide the readers with numerical examples to illustrate some of the established theoretical bounds. In Appendix~\ref{SubSec:PressureVelocity}, we discuss how the developed theoretical framework extends to the pressure-velocity formulations of these models.

\section{Theoretical preliminaries}
In this section, we introduce the necessary theoretical tools to be used in the numerical analysis. We conduct the analysis of mixed-formulation \eqref{eq:continuous_problem} under the following assumptions on the constant medium parameters:
\[c>0, \quad b>0, \quad k, \sigma \in \R .\]

The presence of the sound diffusivity ($b>0$) and thus the strong damping $-b \nabla \cdot \bvt$ contributes to the parabolic-like character of the Kuznetsov equation (see, e.g.,~\cite{kaltenbacher2011well,mizohata1993global} for its analysis), and is essential for the validity of the estimates in Sections~\ref{sec:lin} and~\ref{sec:nonlin_kuz} (see condition \eqref{eq:b_ref} below). However, in Section~\ref{sec:inv_West}, we establish sufficient conditions under which the convergence analysis can be made uniform in $b$ for a simplified setting where $\sigma=0$. 

\color{black}

\subsection{Discretization spaces} 
We assume that $\Om \subset \R^d$ is a polygonal convex domain, where $d \in \{2, 3\}$, so that it can be discretized exactly.
Let $\mathcal{T}_h$ be a regular family of partitions of $\Om$, which satisfy the quasi-uniformity condition. 
Using partition-compatible spaces of approximation (see e.g, \cite[Ch. 2]{boffi2013mixed}),
we construct a collection of finite dimensional subspaces $\spaceS\times \spaceV$ of $\Ltwo \times \Hdiv $ of order $p \geq 1$ ($M_p$ in the notation of \cite{boffi2013mixed}) satisfying \[\div(\spaceV) \subseteq \spaceS.\]
To $\spaceS\times \spaceV$, we associate a positive parameter $h$ representing the discretization parameter of $\mathcal{T}_h$.
Note that due to the inclusion condition above, the space $\spaceS$ contains non-smooth functions. Therefore, one should be careful when choosing the interpolation operator on this space. We will make this choice explicit later in the analysis.

For a final time $T>0$ and $f \in L^2(0,T; L^2(\Omega))$, the approximate problem studied in this work is then given by
\begin{subequations} \label{IBVP_approx_Kuznetsov}
	\begin{equation} \label{eq_approx_Kuznetsov}
		\left \{\begin{aligned}
			& \begin{multlined}[t]((1+2k \psi_{ht})\psi_{htt}, \phi_h)_{L^2}-c^2 (\nabla \cdot (\bvh+\tfrac{b}{c^2}\bvht), \phi_h)_{L^2}\\+2 \sigma(\boldsymbol{v}_h \cdot \bvht, \phi_h)_{L^2}= (f, \phi_h)_{L^2}, \end{multlined}\\[2mm]
			& (\bvh, \boldsymbol{w}_h)_{L^2}+(\psi_h, \nabla \cdot \boldsymbol{w}_h)_{L^2}=0,
		\end{aligned} \right.
	\end{equation} 
	for all $(\phi_h, \bwh) \in \spaceS \times \spaceV$ a.e.\ in $(0,T)$, supplemented with approximate initial conditions 
	\begin{equation}\label{approx_initcond}
		(\ph, \pht)\vert_{t=0}=(\psi_{0h}, \psi_{1h}) \in \spaceS \times \spaceS,
	\end{equation}
\end{subequations}
which will be set through the mixed projection of exact conditions; see \eqref{eq:mixed_proj} and \eqref{eq:init_propL} below for details.  

\subsection{Properties of the interpolation and projection operators} In the upcoming analysis, we will rely on the properties of the interpolation and projection operators, which we recall here.
Let $\Ih$ be the interpolation operator in $\spaceV$ (as defined in \cite[Section 2.5]{boffi2013mixed}) and let $\pi_h$ be the $L^2$-projection operator on $\div(\spaceV) \subseteq \spaceS$. 
For the upcoming theory we need the following properties to hold:
\begin{align} 
	&\div(\Ih \bv) = \pi_h \div (\bv),
	\shortintertext{and there exists a generic constant $C$ independent of $h$ such that}
	&\nLtwo{\bv - \Ih \bv}\leq C h^m \|\bv\|_{H^m},\\
	&\nLtwo{\div(\bv - \Ih \bv)}\leq C h^s \|\div \bv\|_{H^s},
\end{align}
where $1 \leq m \leq p+1$ and $0\leq s \leq p^*$. These hold for 
\color{black} 
$p^* = p+1$ in the case of Raviart--Thomas (RT), and Brezzi--Douglas--Fortin--Marini (BDFM) elements,
 and for $p^* = p$ in the case of Brezzi--Douglas--Marini (BDM) and Brezzi--Douglas--Dur\'an--Fortin (BDDF) elements; see e.g, \cite{boffi2013mixed,chen2005finite} for the construction and properties of such elements.

Following~\cite{johnson1981error, pani2001mixed}, we introduce the mixed projections $(\tildePp{\psi}, \tildePv{\bv}) \in \spaceS\times\spaceV$ of $( \psi,\bv) \in  \Ltwo \times \Hdiv$ as follows:
\begin{equation}\label{eq:mixed_proj}
	\begin{aligned}
		(\nabla \cdot(\bv-\tildePv{\bv}), \phi_h)_{L^2}=&\,0,\quad  &&\forall \phi_h \in \spaceS,\\
		(\bv-\tildePv{\bv}, \boldsymbol{w}_h)+(\psi-\tildePp{\psi}, \nabla \cdot \boldsymbol{w}_h)_{L^2}=&\,0, \quad &&\forall \bwh \in \spaceV.
	\end{aligned}
\end{equation}
We next state a useful lemma for the  analysis  as it deals with the recurring second equation of the mixed problems (including the mixed projection). 
\begin{lemma}\label{lem:elliptic_bounds}
	Let $ 1 \leq q \leq \infty$. There exist constants $C_q>0$ and $C>0$, independent of $h$, such that if $\phi_h \in \spaceS$ and $\bw \in L^2(\Omega)^d$  
	satisfy 
	\[(\bw, \bwh) + (\phi_h, \div \bwh) = 0, \qquad \forall \bwh \in \spaceV, \]
	then 
	\begin{equation}\|\phi_h\|_{L^q} \leq C_q \nLtwo{\bw} \quad \textrm{for }\ \left\{
		\begin{aligned}
		&{1 \leq q \leq 6 \ \ \, \textrm{if } \quad d = 3},\\
		&{1 \leq q < \infty \ \textrm{if } \quad d = 2}. 
		\end{aligned}
	\right.
	\end{equation}
	For $d=2$, we additionally have 
	\[\|\phi_h\|_{L^\infty} \leq C \log\frac1{h}\nLtwo{\bw}.\]
\end{lemma}
\begin{proof}
	The proof is similar to that given in~\cite[Lemma 1.2]{johnson1981error}, where, instead of \cite[1.4 c)]{johnson1981error}, we rely on the bound 
	\[\|\boldsymbol{I}_h\bv \|_{L^2} \leq \|\bv\|_{W^{1,s}},\] with $s \geq \tfrac{2d}{d+2}$. Thus the maximal $L^q$ regularity obtained is given by $q = \tfrac{s}{s-1}$, such that for $d=3$ we obtain $1 \leq q \leq 6$. For $d=2$, the same result as in \cite{johnson1981error} holds.
\end{proof}
With the following result, we generalize the mixed projection estimates given in \cite[Theorem 1.1]{johnson1981error} to the different elements considered in this work. 
\begin{lemma}\label{lem:interpolation_error}
	Let $\psi \in H^{r+1}(\Om)$ and $\nabla \psi = \bv$ a.e.\ in $\Omega$. Then 
	the mixed projection of $(\psi, \bv)$ satisfies the following bounds a.e.\ in time: 
	\begin{equation} \label{mixed_projection}
		\begin{aligned}
			\left \|\psi-\tildePp{\psi}\right\|_{L^2} \leq&\, Ch^r  \left \| \psi\right\|_{H^{r}}, \quad &&2 \leq r \leq p^*, \\
			\left\|\bv-\tildePv{\bv} \right\|_{L^2} \leq&\, C h^r  \left\| \psi \right\|_{H^{r+1}}, \quad &&1 \leq r \leq p+1.\\
				\|\psi-\tildePp{\psi}\|_{L^q} \leq&\, C h^r  \left\| \psi \right\|_{H^{r+1}}, \quad &&1 \leq r \leq p+1, \ \left\{
				\begin{aligned}
					&{1 \leq q \leq 6 \ \ \, \textrm{if } \quad d = 3},\\
					&{1 \leq q < \infty \ \textrm{if } \quad d = 2}. 
				\end{aligned}
				\right.
			\end{aligned}
	\end{equation}
For $d=2$, we additionally have that
\[\|\psi-\tildePp{\psi}\|_{L^\infty} \leq\, C h^r \log \frac1{h} \left\| \psi \right\|_{H^{r+1}}, \quad 1 \leq r \leq p+1.\]
\end{lemma}
\begin{proof}
	The bounds are obtained by applying the same reasoning as in~\cite[Theorem 1.1]{johnson1981error} with the properties of the interpolant $\Ih$ and $L^2$ projection $\pi_h$ given above, and Lemma~\ref{lem:elliptic_bounds}.
\end{proof}
\subsection*{Notation}
We shall frequently use the notation $x \lesssim y$ which stands for $x \leq Cy$, where $C$ is a generic constant that depends on the reference domain $\Omega$ and, possibly, the time $t$, but not on the discretization parameter. To simplify the notation we will omit the time interval when writing norms covering the whole open segment $(0,T)$, for example, \(\|\cdot\|^2_{L^p(H^r)}\) denotes the norm on $L^p(0,T;H^r)$.

\section{\emph{A priori} analysis of a linearized problem}\label{sec:lin}
Our approach in the analysis is based on combining \emph{a priori} bounds for a linearized problem with a fixed-point approach. To this end, we first consider a nondegenerate linearization of \eqref{IBVP_approx_Kuznetsov} by introducing variable coefficients $\alpha_h$ and $\beta_h$. The problem is  to find $(\ph, \bvh): [0,T] \mapsto \spaceS \times \spaceV$, such that
\begin{equation} \label{weak_approx_linear}
	\left \{\begin{aligned}
		& \begin{multlined}[t]((1+2k\alpha_h(x,t))\psi_{htt}, \phi_h)_{L^2}-c^2 (\nabla \cdot (\bvh+\tfrac{b}{c^2}\bvht), \phi_h)_{L^2}\\+2 \sigma(\boldsymbol{\beta}_h(x,t) \cdot \bvht, \phi_h)_{L^2}=(f, \phi_h)_{L^2},\end{multlined}\\[2mm]
		& (\bvh, \boldsymbol{w}_h)_{L^2}+(\psi_h, \nabla \cdot \boldsymbol{w}_h)_{L^2}= 0,
	\end{aligned} \right.
\end{equation} 
for all $(\phi_h, \bwh) \in \spaceS \times \spaceV$ a.e. in time, supplemented by approximate initial conditions \eqref{approx_initcond}. 

\begin{remark}
	Note that adjoint problems for nonlinear acoustic equations (e.g., in the context of PDE-constrained optimization problems) are of the form of \eqref{weak_approx_linear}; see \cite{meliani2021shape}, for example. Thus, the numerical analysis of the above linearized problem can also also be informative about 
	the discretization error of adjoint-based optimization problems.
\end{remark}
\begin{assumptionp}{K1}\label{Assumption_nondeg} 
	We assume that the coefficient $\alpha_h$ is non-degenerate; that is, there exist $\ulal$, $\olal>0$, independent of $h$, such that
	\[
	0 < 1-2|k|\ulal \leq 1+2k\alpha_h(x,t) \leq 1+2|k|\olal, \quad (x,t) \in \Om \times (0,T).
	\]
	Similarly, we assume that there exists $\overline{\beta}>0$, independent of $h$, such that
	\begin{equation}
		\begin{aligned}
			\nLinfLinf{\boldsymbol{\beta}_h} \leq \overline{\beta}.
		\end{aligned}
	\end{equation}
\end{assumptionp}
	
\noindent Introducing bases of $\spaceS$ and $\spaceV$, the semi-discrete problem can be written in the matrix form:
\begin{equation}
	\begin{aligned}
		M(t) \Psi_{tt}-c^2 B V -bB V_t+2\sigma L(t)V_t=&\, F(t), \\
		D V+B^T \Psi=0,
	\end{aligned}
\end{equation}
with $\Psi(0)$ and $\Psi_t(0)$ given. The matrix $M=M(t)$ is positive definite on account of the non-degeneracy assumption of the coefficient $\alpha_h$, and so we can rewrite the first equation as
\begin{equation}
	\begin{aligned}
		\Psi_{tt}-c^2 M(t)^{-1}B V -bM(t)^{-1}B V_t+2\sigma M(t)^{-1}L(t)V_t=&\,M(t)^{-1} F(t).
	\end{aligned}
\end{equation}
Thus, after eliminating $V$, we have a system of second-order linear ODEs with the coefficients and right-hand side in $L^2(0,T)$ which can be reduced to an ODE system of first order. Thus for  $f \in L^2(0,T; L^2(\Omega))$, from~\cite[Theorem 1.8]{agarwal2001fixed} and via a bootstrap argument, we infer that there exists a unique $(\ph, \bvh) \in H^2(0,T; \spaceS) \times H^2(0,T; \spaceV)$ which solves the linearized semi-discrete problem. \\[1mm]
\noindent {\bf Energy functionals.} To formulate the stability result, we introduce the total potential-velocity energy of system \eqref{eq:continuous_mixed_problem} at time $t \in [0,T]$ as
\begin{equation}
	\begin{aligned}
		E[\psi, \bv](t)=E_{\psi}(t)+E_{\bv}(t),
	\end{aligned}
\end{equation}
where the potential component is
\begin{equation}
	\begin{aligned}
		E_{\psi}(t)=\nLtwo{\psi(t)}^2+\nLtwo{\pt(t)}^2+\int_0^t\nLtwo{\ptt(s)}^2\ds
	\end{aligned}
\end{equation}
and the velocity component is
\begin{equation}\label{eq:kinetic_energy}
	\begin{aligned}
		E_{{\bv}}(t)=\nLtwo{\bv(t)}^2+\nLtwo{\bvt(t)}^2.
	\end{aligned}
\end{equation}

\color{black}
For the stability analysis we need to access the energy at $t=0$. While this can be done in a straightforward way for $E_\psi(0)$, evaluating $E_{{\bv}}(0)$ requires an additional step. That is, we need to evaluate the initial values $\bv(0) = \bv_0$ and $\bv_t(0) = \bv_1$. Given ($\psi_0$, $\psi_1$), this can be done using the weak form of the second equation of \eqref{eq:continuous_mixed_problem} 
\begin{align}\label{eq:init_cont}
	\prodLtwo{\bv_i}{\bw} + \prodLtwo{\psi_i}{\div \bw} = 0
\end{align}
for all $\bw \in \Hdiv$ and for $i=1,2$. In particular, $\bv_i = \nabla \psi_i$ a.e.\ provided $\psi_i$ is smooth enough.

Similarly, given ($\psi_{0h}$,$\psi_{1h}$), the values $\bvh(0) = \bv_{0h}$ and $\bv_{ht}(0) = \bv_{1h}$ are found by solving 
\begin{equation}\label{eq:init_data_mix}
	\begin{aligned}
		(\bv_{ih}, \boldsymbol{w}_h)+(\psi_{ih}, \nabla \cdot \boldsymbol{w}_h)_{L^2}=&\,0
	\end{aligned}
\end{equation}
for all $\bwh \in \spaceV$, where $i=1,2$. With this, we can fully define the initial discrete energy $E[\psi_h, \bvh](0)$. In what follows, we show that \eqref{weak_approx_linear} is stable in the energy norm under Assumption~\ref{Assumption_nondeg}.
\begin{proposition}\label{Prop:LinStability} Let $b>0$ and let Assumptions~\ref{Assumption_nondeg} hold. Given $f \in L^2(0,T; L^2)$, the solution $(\ph, \bv_{h})$ of \eqref{weak_approx_linear} satisfies the following stability bound:
	\begin{equation} \label{stability}
		\begin{aligned}
			E[\psi_h, \bvh](t) \lesssim 	E[\psi_h, \bvh](0)+ \nLtwoLtwot{f(s)}^2, \quad t \geq 0.
		\end{aligned}
	\end{equation}
\end{proposition}
\begin{remark}
	The proof of Proposition~\ref{Prop:LinStability} relies on Gronwall's inequality. Subsequently the hidden constant is exponential in time. In the context of nonlinear ultrasound, this is acceptable as the lifespan of such waves is relatively short. In other applications such as viscoelasticity or if the problem is purely linear with variable coefficients, one can derive time-independent energy bounds on the error provided $\alpha_{h}$ and $\boldsymbol{\beta}_h$ are small enough in suitable norms. We also note that the hidden constant in \eqref{stability} tends to $+\infty$ as $b \rightarrow 0^+$.  Uniform discretization in $b$ is discussed in Section~\ref{sec:inv_West}. 
\end{remark}
\begin{proof}
	The proof is based on a non-standard energy analysis. The strategy can be summarized as follows: \[\textup{I}\cdot(\psi_{ht}+\gamma \psi_{htt})+(\textup{II}_t+\gamma \textup{II}_{tt})\cdot c^2 (\bvh+\tfrac{b}{c^2}\bvht);\]
	that is, we multiply the first equation in \eqref{weak_approx_linear} by $\psi_{ht}+\gamma \psi_{htt}$, the time-differentiated second equation by $c^2 (\bvh+\tfrac{b}{c^2}\bvht)$, and the twice time-differentiated second equation by $\gamma c^2 (\bvh+\tfrac{b}{c^2}\bvht)$ with $\gamma>0$. This approach at first results in the following system:
	\begin{equation}
		\left \{\begin{aligned}
			& \begin{multlined}[t]((1+2k\alpha_h(x,t))\psi_{htt}, \psi_{ht}+\gamma \psi_{htt})_{L^2}-c^2 (\nabla \cdot (\bvh+\tfrac{b}{c^2}\bvht), \psi_{ht}+\gamma \psi_{htt})_{L^2}\\
				+2 \sigma(\boldsymbol{\beta}_h(x,t) \cdot \bvht, \psi_{ht}+\gamma \psi_{htt})_{L^2}=(f, \psi_{ht}+\gamma \psi_{htt})_{L^2}, \end{multlined}\\[2mm]
			&  \begin{multlined}[t](\bvht, c^2 (\bvh+\tfrac{b}{c^2}\bvht))_{L^2}+c^2(\psi_{ht}, \nabla \cdot  (\bvh+\tfrac{b}{c^2}\bvht))_{L^2}= 0,\end{multlined}\\[2mm]
			& \begin{multlined}[t]\gamma (\bv_{htt}, c^2 (\bvh+\tfrac{b}{c^2}\bvht))_{L^2}+c^2\gamma(\psi_{htt}, \nabla \cdot  (\bvh+\tfrac{b}{c^2}\bvht))_{L^2}= 0.\end{multlined}
		\end{aligned} \right.
	\end{equation}	
	Summing up these equations and noting that the $\nabla \cdot$ terms cancel out then leads to
	\begin{equation}
		\begin{aligned}
			\begin{multlined}[t]((1+2k\alpha_h(x,t))\psi_{htt}, \psi_{ht}+\gamma \psi_{htt})_{L^2}+2 \sigma(\boldsymbol{\beta}_h(x,t) \cdot \bvht, \psi_{ht}+\gamma \psi_{htt})_{L^2} \\+
				(\bvht+\gamma \bv_{htt}, c^2 (\bvh+\tfrac{b}{c^2}\bvht))_{L^2}=(f, \psi_{ht}+\gamma \psi_{htt})_{L^2}. \end{multlined}
		\end{aligned} 
	\end{equation}
	Choosing $\gamma=b/c^2$, so as to write the term 
	\begin{equation}
		\begin{aligned}
			(\bvht+\gamma \bv_{htt}, c^2 (\bvh+\tfrac{b}{c^2}\bvht))_{L^2} =&\,   \frac{c^2}{2}\ddt \nLtwo{\bvh+\frac{b}{c^2}\bvht}^2 \\
			=&\, \frac{c^2}{2}\ddt \nLtwo{\bvh}^2+\frac{b^2}{2c^2}\ddt \nLtwo{\bvht}^2+b\ddt \prodLtwo{\bvh}{\bvht},
		\end{aligned}
	\end{equation}
	thus eliminating the need to estimate the product $\prodLtwo{\bvh}{\bv_{htt}}$, we arrive at the following energy identity:
	\begin{equation}
		\begin{aligned}
			\begin{multlined}[t] \frac{b}{c^2} \nLtwo{\sqrt{1+2k\alpha_h}\,\ptth}^2+\frac12 \ddt \nLtwo{\pth}^2+\frac{c^2}{2}\ddt \nLtwo{\bvh}^2+\frac{b^2}{2c^2}\ddt \nLtwo{\bvht}^2
				\\
				=-2k\prodLtwo{\alpha_{h}\phtt}{\pht}
				-2\sigma(\boldsymbol{\beta}_h(x,t) \cdot \bvht, \psi_{ht}+\frac{b}{c^2} \psi_{htt})_{L^2}-b\ddt \prodLtwo{\bvh}{\bvht}\\+ (f, \psi_{ht}+\frac{b}{c^2} \psi_{htt})_{L^2}.\end{multlined}
		\end{aligned}
	\end{equation}
	Recalling Assumption~\ref{Assumption_nondeg}, we can estimate the right-hand-side terms, for some arbitrary $\varepsilon>0$, as follows:
	\begin{align}
		&-2k\prodLtwo{\alpha_{h}\phtt}{\pht} \leq \frac{k^2\olal^2}{\varepsilon} \nLtwo{\psi_{ht}}^2 + \varepsilon \nLtwo{\psi_{htt}}^2, \end{align}
	and
	\begin{align}
		-2\sigma\prodLtwo{\boldsymbol{\beta}_h(x,t) \cdot \bvht}{\psi_{ht}+\frac{b}{c^2} \psi_{htt}} 
		& \leq 2|\sigma| \olbe \nLtwo{\boldsymbol{v}_{ht}}\left(\nLtwo{\psi_{ht}} + \frac{b}{c^2} \nLtwo{\psi_{htt}} \right) \\
		& \leq  \sigma^2 \olbe^2 \nLtwo{\boldsymbol{v}_{ht}}^2 \left(1 + \frac{b^2}{\varepsilon c^4} \right) + \nLtwo{\psi_{ht}}^2 +  \varepsilon \nLtwo{\psi_{htt}}^2.
	\end{align}
	Similarly, we estimate 
	\begin{align}
		(f, \psi_{ht}+\frac{b}{c^2} \psi_{htt})_{L^2} &\leq \nLtwo{f} \left( \nLtwo{\pht} + \frac{b}{c^2} \nLtwo{\phtt} \right)\\
		& \leq  \nLtwo{f}^2  \left(\frac14 + \frac{b^2}{4\varepsilon c^4} \right) + \nLtwo{\psi_{ht}}^2 +  \varepsilon \nLtwo{\psi_{htt}}^2.
	\end{align}
	Finally, because we intend to integrate over time, we estimate
	\begin{align}
		&\begin{multlined}[t]
			- b\intT \ddt \prodLtwo{\bvh(s)}{\bvht(s)} \ds = -b \left[\prodLtwo{\bvh(t)}{\bvht(t)} - \prodLtwo{\bvh(0)}{\bvht(0)} \right] \\
			\leq b  \nLtwo{\bvh(t)}\nLtwo{\bvht(t)} + b  \nLtwo{\bvh(0)}\nLtwo{\bvht(0)} \\
			\leq  \frac{b^2}{4\varepsilon} \nLtwo{\bvh(t)}^2 + \varepsilon \nLtwo{\bvht(t)}^2 
			+ \frac{b^2}{2\varepsilon} \nLtwo{\bvh(0)}^2 + \frac{\varepsilon}{2} \nLtwo{\bvht(0)}^2,
		\end{multlined}
		\intertext{which on account of}
		&\nLtwo{\bvh(t)} = \nLtwo{\intT \bvht(s)\ds + \bvh(0)}   \leq \sqrt{t} \left(\intT \nLtwo{\bvht(s)}^2\ds \right)^{1/2}   + \nLtwo{\bvh(0)},
		\intertext{yields}
		&\begin{multlined}[t] 
			-b\intT \ddt \prodLtwo{\bvh(s)}{\bvht(s)} \ds \leq \frac{b^2}{2\varepsilon}\, t \intT \nLtwo{\bvht(s)}^2\ds + \varepsilon \nLtwo{\bvht(t)}^2 \\
			+ \frac{b^2}{\varepsilon} \nLtwo{\bvh(0)}^2 + \frac{\varepsilon}{2} \nLtwo{\bvht(0)}^2.
		\end{multlined}
	\end{align}
	
	\noindent Integrating in time on $(0,t)$, for some $0<t<T$, we obtain the following inequality:
	\begin{equation}
		\begin{aligned} 
			&\begin{multlined}[t]\left(\frac{b}{c^2}  \right.(1-2|k|\ulal)\left.\vphantom{\frac11} - 3 \varepsilon \right) \intT \nLtwo{\ptth(s)}^2 \ds +\frac12 \nLtwo{\pth(t)}^2
				+\frac{c^2}{2} \nLtwo{\bvh(t)}^2
				\\+\left(\frac{b^2}{2c^2}-\varepsilon \right) \nLtwo{\bvht(t)}^2 \end{multlined}
			\\
			\leq&\, \begin{multlined}[t] \left(\frac{k^2\olal^2}{\varepsilon} +2\right)\intT \nLtwo{\psi_{ht(s)}}^2 \ds
				+ \left[\sigma^2 \olbe^2 \left(1 + \frac{b^2}{\varepsilon c^4} \right) + \frac{b^2}{2\varepsilon}\, t\right]\intT \nLtwo{\boldsymbol{v}_{ht}(s)}^2 \ds 
				\\+  \left(\frac14 + \frac{b^2}{4\varepsilon c^4} \right)\intT \nLtwo{f(s)}^2 \ds + \frac12 \nLtwo{\pth(0)}^2
				+\left(\frac{c^2}{2} + \frac{b^2}{\varepsilon}\right)\nLtwo{\bvh(0)}^2\\+\left(\frac{b^2}{2c^2}+\frac{\varepsilon}{2} \right) \nLtwo{\bvht(0)}^2.\end{multlined}
		\end{aligned}
	\end{equation}
	\noindent Setting 
	\begin{equation}\label{eq:b_ref}
		0<\varepsilon < \min(\frac{b}{3c^2} (1-2|k|\ulal), \frac{b^2}{2c^2}),
	\end{equation}
	using Gronwall's inequality, and then using 
	\[\nLtwo{\psi_h(t)} \leq \sqrt{t} \left(\intT \nLtwo{\psi_{ht}(s)}^2\ds \right)^{1/2}   + \nLtwo{\psi_h(0)}\]
	yields the desired estimate.
\end{proof}

We next perform the error analysis for the semi-discrete linearized problem in a general setting, which allows for the coefficients to be approximations of the exact ones. To this end, we assume the exact solution of \eqref{eq:continuous_problem} to be smooth enough in the following sense. Given $r \in \{2, \ldots, p^*\}$, we assume that $\psi$ belongs to 
\begin{equation} \label{solution_space}
	\begin{aligned}
		X_{r+1} = \left\{ \psi \in L^{\infty}(0,T; H_0^1(\Omega) \cap H^{r+1}(\Omega)) \, \right. |& \left. \, \pt \in L^{\infty}(0,T; H_0^1(\Omega) \cap H^{r+1}(\Omega)),\right. \\ &\,\left. \ptt \in L^{2}(0,T; H_0^1(\Omega) \cap H^{r}(\Omega))  \right\}.
	\end{aligned}
\end{equation}
We denote by $\|\cdot\|_{X_{r+1}}$ the norm associated to this space. With this assumed regularity of $\psi$, formulations \eqref{eq:continuous_problem} and \eqref{eq:continuous_mixed_problem} are equivalent when supplemented by the same initial and boundary data. For $r=2$, the global well-posedness of the Dirichlet initial boundary-value problem in $X_{r+1}$ for the Kuznetsov equation with $f=0$ follows by the analysis in~\cite{mizohata1993global}, under the assumption of sufficiently smooth and small data. Higher-order regularity follows by~\cite{kawashima1992global} under stronger regularity and smallness assumptions on the initial conditions and higher-order compatibility of the initial and boundary data. 
\begin{assumptionp}{K2}\label{Assumption_reg} 
	For a given integer $r \in \{2, \ldots, p^*\}$, we assume that the coefficients $\alpha_h$ and $\boldsymbol{\beta}_h$ approximate $\psi_t$ and $\bv$ up to the following accuracy:
	\begin{equation}
		\begin{aligned}
			&\nLtwo{\pt(t)-\alpha_h(t)} \leq C_* h^{r} \|\psi\|_{X_{r+1}},\\
			&\nLtwo{\boldsymbol{v}(t)-\boldsymbol{\beta}_h(t)} \leq C_* h^{r} \|\psi\|_{X_{r+1}},
		\end{aligned}
	\end{equation}
	for all $t\in [0,T)$, where $C_*>0$ does not depend on $h$.
\end{assumptionp}
\noindent We note that the error $(e_{\psi}, \boldsymbol{e}_{\bv}):=(\psi-\ph, \bv-\bvh)$ satisfies
\begin{equation} \label{linear_error_eq}
	\left \{\begin{aligned}
		& \begin{multlined}[t]((1+2k\alpha_h(x,t))e_{\psi tt}, \phi_h)_{L^2}+2k((\pt-\alpha_h)\psi_{tt}, \phi_h)_{L^2}\\-c^2 (\nabla \cdot (\boldsymbol{e}_{\boldsymbol{v}}
			+\tfrac{b}{c^2}\boldsymbol{e}_{\boldsymbol{v}t}), \phi_h)_{L^2} \\+2 \sigma (\boldsymbol{\beta}_h(x,t) \cdot \boldsymbol{e}_{\boldsymbol{v}t}, \phi_h)_{L^2}
			+2\sigma((\bv-\boldsymbol{\beta}_h) \cdot \bvt, \phi_h)_{L^2}=0,\end{multlined}\\[2mm]
		& (\boldsymbol{e}_{\boldsymbol{v}}, \boldsymbol{w}_h)_{L^2}+(e_{\psi}, \nabla \cdot \boldsymbol{w}_h)_{L^2}= 0,
	\end{aligned} \right.
\end{equation} 
for all $(\phi_h, \bwh) \in \spaceS \times \spaceV$ a.e.\ in time with 
\begin{align}
	(e_{\psi}, \boldsymbol{e}_{\bv})(0) &= (\psi_0-\psi_{0h}, \bv_0-\bv_{0h}),\\
	(e_{\psi t}, \boldsymbol{e}_{\bv t})(0) &= (\psi_1-\psi_{1h}, \bv_1-\bv_{1h});
\end{align}
see \eqref{eq:init_cont} and \eqref{eq:init_data_mix} for the way to set approximate initial data. By employing the mixed projection defined in \eqref{eq:mixed_proj}, we split the error as follows:
\begin{equation} \label{split_error}
	\begin{aligned}
		& e_{\psi}=\psi-\ph=(\psi-\tildePp{\psi})+(\tildePp{\psi} -\psi_h):= \tilde{e}_{\psi}+\tilde{e}_{\psi h},\\
		&  \boldsymbol{e}_{\bv}= (\bv-\tildePv{\bv})+(\tildePv{\bv}-\bvh):=\tilde{\boldsymbol{e}}_{\bv}+\tilde{\boldsymbol{e}}_{\boldsymbol{v}h}.
	\end{aligned}
\end{equation}
Thus $(\tilde{e}_{\psi h}, \tilde{\boldsymbol{e}}_{\bv h})$ can be seen as the solution to 
\begin{equation} 
	\left \{\begin{aligned}
		& \begin{multlined}[t]((1+2k\alpha_h(x,t))(\tilde{e}_{\psi tt}+\tilde{e}_{\psi h tt}), \phi_h)_{L^2}+2k((\pt-\alpha_h)\psi_{tt}, \phi_h)_{L^2}\\
			-c^2 (\nabla \cdot (\tilde{\boldsymbol{e}}_{\boldsymbol{v} h}+\tfrac{b}{c^2}\tilde{\boldsymbol{e}}_{\boldsymbol{v} h t}), \phi_h)_{L^2}+2 \sigma(\boldsymbol{\beta}_h(x,t) \cdot \tilde{\boldsymbol{e}}_{\boldsymbol{v}t}+\tilde{\boldsymbol{e}}_{\boldsymbol{v} h t}, \phi_h)_{L^2}\\
			+2\sigma((\bv-\boldsymbol{\beta}_h) \cdot \bvt, \phi_h)_{L^2}=0,\end{multlined}\\[2mm]
		& (\tilde{\boldsymbol{e}}_{\boldsymbol{v}h}, \boldsymbol{w}_h)_{L^2}+(\tilde{e}_{\psi h}, \nabla \cdot \boldsymbol{w}_h)_{L^2}= 0.
	\end{aligned} \right.
\end{equation} 
Equivalently,
\begin{equation} \label{linear_error_proj_eq}
	\left \{\begin{aligned}
		& \begin{multlined}[t]((1+2k\alpha_h(x,t))\tilde{e}_{\psi h tt}, \phi_h)_{L^2}-c^2 (\nabla \cdot (\tilde{\boldsymbol{e}}_{\boldsymbol{v} h}+\tfrac{b}{c^2}\tilde{\boldsymbol{e}}_{\boldsymbol{v} h t}), \phi_h)_{L^2}\\+2 \sigma(\boldsymbol{\beta}_h(x,t) \cdot \tilde{\boldsymbol{e}}_{\boldsymbol{v} h t}, \phi_h)_{L^2}=(\tilde{f}, \phi_h)_{L^2},\end{multlined}\\[2mm]
		& (\tilde{\boldsymbol{e}}_{\boldsymbol{v}h}, \boldsymbol{w}_h)_{L^2}+(\tilde{e}_{\psi h}, \nabla \cdot \boldsymbol{w}_h)_{L^2}= 0
	\end{aligned} \right.
\end{equation} 
with the right-hand side given by
\[
\tilde{f}=-(1+2k\alpha_h)\tilde{e}_{\psi tt}-2k(\pt-\alpha_h)\psi_{tt}
-2\sigma \boldsymbol{\beta}_h(x,t) \cdot \tilde{\boldsymbol{e}}_{\boldsymbol{v}t}-2\sigma(\bv-\boldsymbol{\beta}_h)\cdot \bv_t.
\]
By the stability result of Proposition~\ref{Prop:LinStability}, we immediately have
\begin{equation} \label{energy_error}
	\begin{aligned}
		E[\tilde{e}_{\psi h}, \tilde{\boldsymbol{e}}_{\boldsymbol{v}h}](t) \lesssim 	E[\tilde{e}_{\psi h}, \tilde{\boldsymbol{e}}_{\boldsymbol{v}h}](0)+\int_0^t \nLtwo{\tilde{f}(s)}^2\ds, \quad t \geq 0,
	\end{aligned}
\end{equation}
which allows us to state the following error estimate for the linearized problem.
\begin{proposition}\label{Prop:LinL2} Let $b>0$. Let $2 \leq r \leq p^*$, and Assumptions~\ref{Assumption_nondeg} and \ref{Assumption_reg} hold. Furthermore, let $\psi \in X_{r+1}$ be the solution of the exact problem \eqref{eq:continuous_problem} with a sufficiently smooth source term $f$, and coupled with homogeneous Dirichlet data and suitable initial conditions $(\psi(0),\psi_t(0))=(\psi_0,\psi_1)$ and let $v = \nabla \psi$. 	
	Let the approximate initial conditions of \eqref{IBVP_approx_Kuznetsov} be set through the mixed projection
	\begin{equation}\label{eq:init_propL}
		\begin{aligned}
			(\psi_{0h}, \bv_{0h})=(\tildePp{\psi_0}, \tildePv{\bv_0}), \\
			(\psi_{1h}, \bv_{1h})=(\tildePp{\psi_1}, \tildePv{\bv_1}),
		\end{aligned} 
	\end{equation}
	where $v_0$ and $v_1$ are obtained through \eqref{eq:init_cont}.
	Then the solution $(\ph, \bvh)$ of 
	\eqref{weak_approx_linear} satisfies the following bound:
	\begin{equation} \label{energy_error3} 
		\begin{aligned}
			&\begin{multlined}[t] 
				\nLtwo{\psi(t)-\ph(t)}^2+\nLtwo{\pt(t)-\pth(t)}^2
				+\int_0^t \nLtwo{\ptt(s)-\phtt (s)}^2 \ds \\ +\nLtwo{\bv(t)-\bvh(t)}^2+\nLtwo{\bv_t(t)-\bvht(t)}^2 \lesssim\, h^{2r} \|\psi\|^2_{X_{r+1}} \end{multlined}
		\end{aligned}
	\end{equation}
	for $t \in [0,T]$, where the hidden constant has the following form:
	\[
	C_{\textup{lin}} = C(T) \left(1+C_*^2\left(k^2 \|\psi_{tt}\|_{L^2(0,T; L^\infty)}^2+ \sigma^2 \|\bvt\|_{L^2(0,T; L^\infty)}^2\right)\right).
	\]
\end{proposition}
\begin{remark}[Equivalence of semi-discrete formulations]
	The choice of approximate initial data \eqref{eq:init_propL} aligns with the choice made for the stability analysis and in the preceding discussion. Indeed, subtracting the second equation of the mixed projection:
	\begin{equation}
		\left\{	\begin{aligned}
			(\nabla \cdot(\bv_i-\bv_{ih}), \phi_h)_{L^2}=&\,0,\quad  &&\forall \phi_h \in \spaceS,\\
			(\bv_i-\bv_{ih}, \boldsymbol{w}_h)+(\psi_i-\psi_{ih}, \nabla \cdot \boldsymbol{w}_h)_{L^2}=&\,0, \quad &&\forall \bwh \in \spaceV,
		\end{aligned}
		\right.
	\end{equation}
	from \eqref{eq:init_cont} yields \eqref{eq:init_data_mix} for $i=1,2$.
	
\end{remark}	
\begin{proof}
	With our choice of the approximate initial data, we know that
	\[
	\tilde{e}_{\psi h}(0)= \tildePp{\psi}(0)-\psi_h(0)=0, \qquad \tilde{e}_{\psi h t}(0)= \tildePp{\psi_t}(0)-\psi_{ht}(0)=0.
	\]
	Furthermore, by choosing $\boldsymbol{w}_h=\tilde{\boldsymbol{e}}_{\boldsymbol{v}h}(0)$ in
	\[
	(\tilde{\boldsymbol{e}}_{\boldsymbol{v}h}, \boldsymbol{w}_h)_{L^2}+(\tilde{e}_{\psi h}, \nabla \cdot \boldsymbol{w}_h)_{L^2}= 0, \ t \in [0,T],
	\]
	and proceeding similarly in the time-differentiated equation, we have
	\[
	\tilde{\boldsymbol{e}}_{\bv h}(0)= \tildePv{\bv}(0)-\bv_h(0)=0, \qquad \tilde{\boldsymbol{e}}_{\bv h t}(0)= \tildePv{\bv_t}(0)-\bv_{ht}(0)=0.
	\]
	Therefore, 
	\[
	E[\tilde{e}_{\psi h}, \tilde{\boldsymbol{e}}_{\boldsymbol{v}h}](0)=0
	\]
	and the bound \eqref{energy_error} reduces to
	\begin{equation} \label{eq:red_energy_error}
		\begin{aligned}
			E[\tilde{e}_{\psi h}, \tilde{\boldsymbol{e}}_{\boldsymbol{v}h}](t) \lesssim  \int_0^t \nLtwo{\tilde{f}(s)}^2\ds, \quad t \geq 0.
		\end{aligned}
	\end{equation}	
	We can estimate $\tilde{f}$ by relying on H\"older's inequality	
	\begin{equation}
		\begin{aligned}
			\|\tilde{f}\|_{L^2(0,t; L^2)} & \leq \begin{multlined}[t] (1+2|k|\olal) \|\tilde{e}_{\psi tt}\|_{L^2(0,t; L^2)}+2|k|\|\psi_t-\alpha_h\|_{L^\infty(0,t; L^2)}\|\psi_{tt}\|_{L^2(0,t; L^\infty)}\\
				+2|\sigma|\overline{\beta}\|\tilde{\boldsymbol{e}}_{\bv t}\|_{L^2(0,t; L^2)} +  2|\sigma|\|\bv-\boldsymbol{\beta}_h\|_{L^\infty(0,t;L^2)}\|\boldsymbol{v}_t\|_{L^2(0,t; L^\infty)}\end{multlined}\\
			& \lesssim \begin{multlined}[t]  h^r \|\psi \|_{X_{r+1}}+ |k| C_* h^r \|\psi\|_{X_{r+1}}\|\psi_{tt}\|_{L^2(0,t; L^\infty)}+ C_* h^r \|\psi\|_{X_{r+1}}\|\bvt\|_{L^2(0,t; L^\infty)},\end{multlined}
		\end{aligned}
	\end{equation}
	where, in the last line, we have used the approximation properties of the mixed projection given in \eqref{mixed_projection}, the regularity of the solution to the continuous problem, and Assumptions~\ref{Assumption_nondeg} and \ref{Assumption_reg}. 
	
	Additionally, a bound on $\nLtwo{\psi(t)-\psi_{h}(t)}$ can be obtained using
	\[
	\nLtwo{\tildePp\psi(t)-\psi_h(t)}=\left \|\int_0^t (\tildePp\psi_t(s)-\psi_{ht}(s))\ds \right\|_{L^2} \leq T\sup_{s\in (0,t)} \nLtwo{\tildePp\psi_t(s)-\psi_{ht}(s)},
	\]	
	which completes the proof.
\end{proof}

Note that in the special case $\sigma = k = 0$, we recover a bound for the strongly damped linear wave equation. Compared to the available results in the literature \cite[Theorem 2.1]{pani2001mixed}, we impose a more regular exact solution but also obtain error bounds in higher-order norms with respect to time.

\section{\emph{A priori} error analysis of the Kuznetsov equation in mixed form}\label{sec:nonlin_kuz}
We are now ready to analyze the semi-discrete Kuznetsov equation in the potential-velocity form \eqref{IBVP_approx_Kuznetsov}. Given $r \in \{2, \ldots, p^*\}$ and approximate initial data $(\psi_{0h}, \psi_{1h}) \in \spaceS \times \spaceS$, we introduce the ball
\begin{equation}
	\begin{aligned}
		\BK=\left\{\vphantom{\int_0^t}\right.&(\psi_h^*, \bv^*_h) \in H^2(0,T; \spaceS) \times C^1([0,T]; \spaceV):  \
		(\psi_h^*, \psi^*_{ht},\bv_h^*, \bv^*_{ht})\vert_{t=0} = (\psi_{0h}, \psi_{1h}, \bv_{0h},\bv_{1h}),
	\\& \sup_{t \in (0,T)} E[\psi-{\psi}^*_h, \bv-\bvh^*](t) \leq C_{\textup{nl}}^2 h^{2r} \|\psi\|^2_{X_{r+1}}	\left.\vphantom{\int_0^t}\right\},
	\end{aligned}
\end{equation}
where $C_{\textup{nl}}>0$ will be specified in the upcoming analysis. 
The proof follows by employing the Banach fixed-point theorem to the mapping 
\[\mathcal{F}: \BK \ni (\psi_h^*, \bv^*_h) \mapsto (\psi_h, \bv_h), \]
with $ (\psi_h, \bv_h)$ being the solution of the linear problem \eqref{weak_approx_linear}, where we choose \[\alpha_h=  \psi_{ht}^* \ \textrm{and } \boldsymbol{\beta}_h= \boldsymbol{v}_h^*.\] Recall that the solution space $X_{r+1}$ for the exact potential is defined in \eqref{solution_space}.

\begin{theorem} \label{Thm:Kuzn} Let $b>0$. Let $2 \leq r \leq p^*$ and 
	let $\psi \in X_{r+1}$ be the solution of the exact problem \eqref{eq:continuous_problem} with a sufficiently smooth source term $f$, and coupled with homogeneous Dirichlet data and suitable initial conditions $(\psi(0),\psi_t(0))=(\psi_0,\psi_1)$ and let $v = \nabla \psi$. 
	Furthermore, let the approximate initial data $(\psi_{0h}, \psi_{1h})$ be chosen as in Proposition~\ref{Prop:LinL2}. Then there exist \[\overline{h} = \overline{h}(\|\psi\|_{X_{r+1}})<1 \quad \text{and} \quad M = M(k,\sigma, T)>0,\] such that for $0<h<\overline h$
	and 
	\begin{equation} \label{Linf_smallness}
		\sup_{t \in (0,T)}\|\bv(t)\|^2_{L^{\infty}}+\sup_{t \in (0,T)}\|\pt(t)\|^2_{L^\infty} +\int_0^T \left( \|\ptt(s)\|_{L^\infty}^2 + \|\bvt(s)\|^2_{L^\infty} \right)\ds \leq M,
	\end{equation}
	there is a unique $(\psi_h, \bv_h)\in \BK$, which solves \eqref{IBVP_approx_Kuznetsov}. 
\end{theorem}
Before going into the proof, we will need to define an appropriate interpolation operator on $\spaceS$. Notice that due to the condition $ p^*\geq 2$, the space $\spaceS$ contains the nodal interpolants, which we denote hereafter by $I_h$. These are suitable, since the quantities concerned ($\psi$, $\psi_t$, $\psi_{tt}$) are in $H^r(\Om)$ with $r\geq2$ and we can thus rely on the embedding $H^r(\Om) \hookrightarrow C(\Om)$; see~\cite[Theorem 12]{burenkov1998sobolev}. 
\color{black}
\begin{proof}
	To be able to employ Banach's fixed-point theorem on $\mathcal{F}$, we first check that the assumptions of Proposition~\ref{Prop:LinStability} hold. Using the properties of the interpolation operator, we find that
	\begin{equation}
		\begin{aligned}
			\nLinfLinf{\alpha_h} \leq&\, \|\psi_{ht}^*-I_h \psi_t\|_{L^\infty(L^\infty)}+ \|I_h \psi_t\|_{L^\infty(L^\infty)}
			\\
			\leq&\, \overline{h}^{-d/2} \|\psi_{ht}^*-I_h \psi_t\|_{L^\infty(L^2)}+\|I_h \psi_t\|_{L^\infty(L^\infty)}\\
			\leq&\, \overline{h}^{-d/2} \|\psi_{ht}^*-\psi_t\|_{L^\infty(L^2)}+ \overline{h}^{-d/2} \|\psi_{t}-I_h \psi_t\|_{L^\infty(L^2)} +\|I_h \psi_t\|_{L^\infty(L^\infty)}.
		\end{aligned}
	\end{equation}
	Thus, since $(\psi_h^*, \bvh^*) \in \BK$, we can guarantee that
	\begin{equation}\label{eq:control_bound_discussion}
		\begin{aligned}
			\ulal=\olal = & C \left((1+C_\textup{nl})\overline{h}^{r-d/2}  \|\psi\|_{X_{r+1}} + M^{1/2}\right) \in (0,1) 
		\end{aligned}
	\end{equation}
	for sufficiently small $M$ and $\overline h$, with $C$ being independent of $h$ and $M$. We can also guarantee uniform boundedness of $\boldsymbol{\beta}_h$ since
	\begin{equation}
		\begin{aligned}
			\|\boldsymbol{\beta}_h\|_{L^\infty(L^\infty)} \leq&\,  \|\bv_{h}^*-\boldsymbol{I}_h \bv\|_{L^\infty(L^\infty)}+\|\boldsymbol{I}_h \bv\|_{L^\infty(L^\infty)}\\
			\leq&\,  \overline{h}^{-d/2} \|\bv_{h}^*-\boldsymbol{I}_h \bv\|_{L^\infty(L^2)}+\|\boldsymbol{I}_h \bv\|_{L^\infty(L^\infty)}.
		\end{aligned}
	\end{equation}
	Thus, the set $\BK$ is non-empty as the solution of the linear problem belongs to it, provided
	\begin{equation} \label{error_fixed_point_iteration}
		\sup_{t \in (0,T)} E[\psi-{\psi}_h, \bv-\bvh](t) \leq C_{\textup{nl}}^2 h^{2r} \|\psi\|^2_{X_{r+1}},
	\end{equation}
	which holds as long as 
	\[
	C(T) \left(1+C_\textup{nl}^2 (k^2 \|\psi_{tt}\|_{L^2(0,T; L^\infty)}^2+ \sigma^2 \|\bvt\|_{L^2(0,T; L^\infty)}^2)\right) \leq 
	C(T) \left(1+C_\textup{nl}^2 \left(k^2 M+ \sigma^2 M \right)\right) \leq  C_\textup{nl}^2.
	\]
	This can be ensured by choosing
	\[ C_\textup{nl}^2 \geq \frac1{\frac1{ C(T)}-k^2 M - \sigma^2 M}\]
	with $M$ small enough so that the denominator is positive. The mapping $\mathcal{F}$ is then well-defined and, on account of estimate \eqref{error_fixed_point_iteration}, $\mathcal{F}(\BK) \subset \BK$. \\
	\indent We next prove strict contractivity of $\mathcal{F}$. Let $(\psi_h^{*(1)}, \bvh^{*(1)})$, $(\psi_h^{*(2)}, \bvh^{*(2)}) \in \BK$. Denote 
	\[
	\begin{aligned}
		(\psi^{(1)}_h, \bvh^{(1)})=&\, \mathcal{F}(\psi_h^{*(1)}, \bvh^{*(1)}), \quad (\psi^{(2)}_h, \bvh^{(2)})=\mathcal{F}(\psi_h^{*(2)}, \bvh^{*(2)}),
	\end{aligned}
	\]
	and the differences
	\[
	\begin{aligned}
		\opsi_h^*=&\,\psi_h^{*(1)}-\psi_h^{*(2)},\quad \obvh^*=\bvh^{*(1)}- \bvh^{*(2)}, \\
		\opsi_h=&\, \psi_h^{(1)}-\psi_h^{(2)},\quad  \obvh=\bvh^{(1)}- \bvh^{(2)}.
	\end{aligned}
	\]
	Then $(\opsi_h, \obvh)$ solves the problem
	\begin{equation} 
		\left \{\begin{aligned} \label{eq:difference_eq}
			& \begin{multlined}[t]((1+2k \psi_{ht}^{*(1)})\opsi_{htt}, \phi_h)_{L^2}-c^2 (\nabla \cdot (\obvh+\tfrac{b}{c^2}\obvht), \phi_h)_{L^2}\\+2 \sigma(\bvh^{*(1)} \cdot \obvht, \phi_h)_{L^2}= (\tilde{f}, \phi)_{L^2},\end{multlined}\\[2mm]
			& (\obvh, \boldsymbol{w}_h)_{L^2}+(\opsi_h, \nabla \cdot \boldsymbol{w}_h)_{L^2}=0
		\end{aligned} \right.
	\end{equation}
	for all $(\phi_h, \bw) \in \spaceS\times \spaceV$,
	with the right-hand side 
	\begin{align}\label{eq:source_diff}
	\tilde{f}=-2k \opsi^*_{ht}\psi^{(2)}_{htt}-2\sigma \obvh^* \cdot \bvht^{(2)},
	\end{align}
	and zero initial data. The stability analysis yields
	\begin{equation} \label{energy_error2}
		\begin{aligned}
			E[\opsi_h, \obvh](t) \lesssim \int_0^t \nLtwo{\tilde{f}(s)}^2\ds, \quad t \geq 0.
		\end{aligned}
	\end{equation}
	We can then estimate $\tilde{f}$ as follows:
	\begin{align}
		\intT \|\tilde{f}\|_{L^2}^2 \dt & \leq 2 |k| \|\psi^{(2)}_{htt}\|_{L^2(L^\infty)} \|\opsi^*_{ht}\|_{L^\infty(L^2)} 
		+ 2 |\sigma| \|\bvht^{(2)}\|_{L^2(L^\infty)} \|\obvh^*\|_{L^\infty(L^2)},
		\intertext{from which we infer}
		\intT \|\tilde{f}\|_{L^2}^2 \dt & \leq 4 \left(|k|\intT \|\psi^{(2)}_{htt}\|_{L^\infty}^2 \dt + |\sigma|\intT \|\bvht^{(2)}\|_{L^\infty}^2\dt \right) \sup_{\tau \in (0,t)} E[\opsi^*_h, \obvh^*](\tau).
	\end{align}
	
	We can estimate the terms within the brackets uniformly with respect to $h$. Indeed using the following uniform bounds:
	\begin{align}\label{eq:error_est_fixed}
		\begin{split}
			\|\psi^{(2)}_{htt}\|_{L^2(L^\infty)} \lesssim&\,  \overline{h}^{-d/2} \|\psi_{htt}^{(2)}-I_h \psi_{tt}\|_{L^2(L^2)}+\|I_h \psi_{tt}\|_{L^2(L^\infty)}\\
			\lesssim &\,  \overline{h}^{-d/2} \|\psi_{htt}^{(2)}- \psi_{tt}\|_{L^2(L^2)} +  \overline{h}^{-d/2} \|\psi_{tt}-I_h \psi_{tt}\|_{L^2(L^2)}+\|I_h \psi_{tt}\|_{L^2(L^\infty)}\\
			\lesssim &\, \overline{h}^{r-d/2} \|\psi\|_{X_{r+1}} + M^{1/2},
		\end{split}
	\end{align}
and  
	\begin{align}\label{eq:error_est_fixed2}
		\|\bvht^{(2)}\|_{L^2(L^\infty)} \lesssim&\,   \overline{h}^{-d/2} \|\bv_{ht}^{(2)}-\boldsymbol{I}_h \bvt\|_{L^2(L^2)}
		+\|\boldsymbol{I}_h \bvt\|_{L^2(L^\infty)}
		\\
		\lesssim&\,\overline{h}^{r-d/2} \|\psi\|_{X_{r+1}} + M^{1/2},
	\end{align}
	we obtain strict contractivity by additionally reducing $M>0$ and $\overline{h}>0$.
\end{proof}
We note that the smallness condition \eqref{Linf_smallness} on the potential can be guaranteed via the Sobolev embedding through the smallness of its $\|\cdot\|_{X_{r+1}}$ norm, which in turn can be ensured through sufficiently  small data in a suitable topology; see \cite{kawashima1992global, mizohata1993global} for the well-posedness analysis. However, it might be possible to impose the smallness in a lower-order topology than that dictated by the solution space by using a suitable interpolation inequality (such as Agmon's inequality~\cite[Ch.\ 13]{agmon1959estimates}) instead of an embedding in the course of the well-posedness analysis. For such approaches in the analysis of nonlinear acoustic equations, see~\cite{bongarti2021vanishing, kaltenbacher2022parabolic}. Thus \eqref{Linf_smallness} may be a more realistic theoretical constraint than imposing smallness of $\|\psi \|_{X_{r+1}}$, as the exact ultrasound data is in practice often smooth but not necessarily small in higher-order norms; see, e.g.,~\cite{nikolic2019priori, kaltenbacher2007numerical}.

\subsection{\emph{A priori} $L^q$ error estimates for the Kuznetsov equation}
In this section, following the general approach of~\cite{pani2001mixed}, we derive error estimates in the $L^q(\Omega)$ norm for the involved scalar quantities, where $q$ depends on the spatial dimension. If we restrict ourselves to a domain in $\R^2$, then we also have an \emph{a priori} error estimate in the maximum error norm. 

\begin{theorem}\label{prop:Linf}
	Let the assumptions of Theorem~\ref{Thm:Kuzn} hold and let $(\psi_h, \bv_h) \in \BK$ be the unique couple which solves \eqref{IBVP_approx_Kuznetsov}. Then the following estimate holds for all $t \in [0,T]$:
	\begin{equation} \label{W1q_bound}
		\begin{aligned}
			&\|\psi(t)-\ph(t)\|_{L^q}  + \|\psi_t(t)-\psi_{ht}(t)\|_{L^q} \lesssim \,  h^r \|\psi\|_{X_{r+1}} \quad \textrm{for }\ \left\{
			\begin{aligned}
				&{1 \leq q \leq 6 \ \ \, \textrm{if } \quad d = 3},\\
				&{1 \leq q < \infty \ \textrm{if } \quad d = 2}. 
			\end{aligned}
			\right.
		\end{aligned}
	\end{equation}
If $d=2$, then
\begin{equation} \label{W1inf_bound}
	\begin{aligned}
		&\|\psi(t)-\ph(t)\|_{L^\infty}  + \|\psi_t(t)-\psi_{ht}(t)\|_{L^\infty} \lesssim \,  h^r \log\frac1{h} \|\psi\|_{X_{r+1}}.
	\end{aligned}
\end{equation}
\end{theorem}
\begin{proof}
	Let $\psi$ be the solution of the Kuznetsov equation \eqref{eq:continuous_mixed_problem} coupled with appropriate boundary and initial conditions and let $\bv=\nabla \psi$. 
	We again use the splitting of the error \eqref{split_error}, where $(\psi_h, \bv_h)$ solves the approximate nonlinear problem. Then $(\tilde{e}_{\psi h}, \tilde{\boldsymbol{e}}_{\boldsymbol{v}h})$ satisfies
	\begin{equation} \label{nonlinear_error_eq}
		\left \{\begin{aligned}
			& \begin{multlined}[t]((1+2k\psi_{ht})\tilde{e}_{\psi h tt}, \phi_h)_{L^2}-c^2 (\nabla \cdot (\tilde{\boldsymbol{e}}_{\boldsymbol{v}h}
				+\tfrac{b}{c^2}\tilde{\boldsymbol{e}}_{\boldsymbol{v}ht}), \phi_h)_{L^2} \\
				=- ((1+2k\psi_{ht})\tilde{e}_{\psi tt}, \phi_h)_{L^2} + 2k(e_{\psi t}\psi_{tt}, \phi_h)_{L^2}
				+ 2\sigma(\boldsymbol{e}_{\boldsymbol{v}} \cdot \bvt, \phi_h)_{L^2}
				\\+ 2 \sigma (\bvh \cdot \boldsymbol{e}_{\boldsymbol{v}t}, \phi_h)_{L^2},
			\end{multlined} 
			\\[2mm]
			& (\boldsymbol{e}_{\boldsymbol{v}}, \boldsymbol{w}_h)_{L^2}+(e_{\psi}, \nabla \cdot \boldsymbol{w}_h)_{L^2}= 0
		\end{aligned} \right.
	\end{equation} 
	a.e.\ in time. Since $(\psi_h,\bvh) \in \BK$, we know that 
	\[E[\tilde{e}_{\psi h}, \tilde{\boldsymbol{e}}_{\boldsymbol{v}h}](t) \leq C h^{2r} \|\psi\|^2_{X_{r+1}}.\]
	We  use Lemma~\ref{lem:elliptic_bounds} with the second equation in \eqref{nonlinear_error_eq} to estimate 
	\begin{align} 
		{\|e_{\psi}(t)\|_{L^q}} \lesssim\; & \nLtwo{\boldsymbol{e}_{\boldsymbol{v}}(t)}\\
		\lesssim\; &  \nLtwo{\tilde{\boldsymbol{e}}_{\boldsymbol{v}}(t)} + \nLtwo{\tilde{\boldsymbol{e}}_{\boldsymbol{v} h}(t)};
	\end{align}
	{for appropriate ranges of $q$ (as given in the statement of the theorem)}. Further, using approximation properties \eqref{mixed_projection} of the mixed projection, we can estimate 
	\[\nLtwo{\tilde{\boldsymbol{e}}_{\boldsymbol{v}}(t)} \lesssim h^r \|\psi(t)\|_{H^{r+1}},
	\]
	while 
	\[\nLtwo{\tilde{\boldsymbol{e}}_{\boldsymbol{v} h}(t)} \leq  \sqrt{E[\tilde{e}_{\psi h}, \tilde{\boldsymbol{e}}_{\boldsymbol{v}h}](t)} \leq C^{1/2} h^r \|\psi\|_{X_{r+1}}. \]
	We can reason the same way to additionally estimate {$\|e_{\psi t}(t)\|_{L^q}$} and arrive at bound  \eqref{W1q_bound}.  The proof of \eqref{W1inf_bound} follows along similar lines, so we omit the details here.
\end{proof}

\begin{remark}[Analysis of the strongly damped Westervelt equation in mixed form]
	By setting $\sigma = 0$ and adjusting the coefficient $k$, from Theorems~\ref{Thm:Kuzn} and~\ref{prop:Linf}, we directly obtain \emph{a priori} estimates for the damped Westervelt equation in potential-velocity form:
	\begin{equation}
		\left \{ \begin{aligned}
			&(1+2\tilde{k} \psi_t)\psi_{tt}-c^2 \nabla \cdot \bv -b \, \nabla \cdot \bvt=f,\\[2mm]
			& \boldsymbol{v}= \nabla \psi,
		\end{aligned}\right.
	\end{equation} 
	with $\tilde{k}= k+1/c^2.$ Results for the strongly damped linear wave equation follow by additionally setting $\tilde{k}=0$. 
\end{remark}

\section{Uniform discretization of the Westervelt equation in mixed form}\label{sec:inv_West}
The bounds established above are not robust with respect to the sound diffusivity $b$; that is, the hidden constant tends to $+\infty$ as $b \rightarrow 0^+$. In this section, we investigate the conditions under which the $L^2$ bounds can be made robust with respect to $b$ in the case  of the Westervelt equation ($\sigma=0$):
\begin{equation}
	\left \{ \begin{aligned}
		&(1+2\tilde{k} \psi_t)\psi_{tt}-c^2 \nabla \cdot \bv {- b \nabla \cdot \bvt}  =f,\\[2mm]
		& \boldsymbol{v}= \nabla \psi.
	\end{aligned}\right.
\end{equation}
\indent We refer to~\cite{kaltenbacher2022parabolic} for the uniform well-posedness analysis of the damped Westervelt equation for sufficiently smooth and small data and short enough final time. Analogously to before, the linearized semi-discrete problem is  to find $(\ph, \bvh): [0,T] \mapsto \spaceS \times \spaceV$, such that
	\begin{equation} \label{weak_approx_linear_inviscid}
		\left \{\begin{aligned}
			& \begin{multlined}[t]((1+2\tilde{k}\alpha_h(x,t))\psi_{htt}, \phi_h)_{L^2}- (c^2 \nabla \cdot \bvh {+ b \nabla \cdot \bvht}, \phi_h)_{L^2}=(f, \phi_h)_{L^2},\end{multlined}\\[2mm]
			& (\bvh, \boldsymbol{w}_h)_{L^2}+(\psi_h, \nabla \cdot \boldsymbol{w}_h)_{L^2}= 0,
		\end{aligned} \right.
	\end{equation} 
	for all $(\phi_h, \bwh) \in \spaceS \times \spaceV$ a.e. in time, supplemented by approximate initial conditions \eqref{approx_initcond}. Provided Assumption~\ref{Assumption_nondeg} on $\alpha_h$ holds and  $f \in L^2(0,T; L^2)$, the solution $(\ph, \bv_{h})$ of the above linearized problem can be shown to satisfy the following estimate for all $t \in [0,T]$:
\begin{equation}\label{eq:Wes_contractivity}
	\begin{aligned}
		E^{\textup{low}}_{\textup{W}}[\psi_h, \bvh](t) := \; & \nLtwo{\psi_{h}(t)}^2+ \nLtwo{\psi_{ht}(t)}^2 + \nLtwo{\bvh(t)}^2 {+ b \intT \nLtwo{\bvht(s)}^2 \ds}  \\ \lesssim \; & E^{\textup{low}}_{\textup{W}}[\psi_h, \bvh](0) + \int_0^t \|f(s)\|^2_{L^2}\ds.
	\end{aligned}
\end{equation}
{where the hidden constant does not depend on $b$.}
This estimate is obtained in a straightforward way using the testing 
\begin{equation}
	\left\{ \begin{aligned}
		& ((1+2\tilde{k}\alpha_h(x,t))\psi_{htt}, \psi_{ht})_{L^2}-(c^2 \nabla \cdot \bvh + {b \nabla \cdot \bvht}, \psi_{ht})_{L^2}
		=(f, \psi_{ht})_{L^2},\\
		& (\bvht, {c^2 \bvh + b \bvht})_{L^2}+(\psi_{ht}, \nabla \cdot  ({c^2 \bvh + b \bvht}))_{L^2}= 0.
	\end{aligned}
	\right.
\end{equation}

However, \eqref{eq:Wes_contractivity} would not allow us to infer a stability bound for the nonlinear problem as we need to be able to control  $\psi_{htt}$ in the fixed-point argument (similarly to the needed bound on \eqref{eq:source_diff} for the  damped Kuznetsov equation). To this end, the idea in the uniform analysis is to consider the time-differentiated first equation in  \eqref{weak_approx_linear_inviscid} and the twice time-differentiated second equation with the following testing strategy:
\begin{equation}
	\left \{\begin{aligned}
		& \begin{multlined}[t]((1+2\tilde{k}\alpha_h(x,t))\psi_{httt}, \psi_{htt})_{L^2}-(c^2 \nabla \cdot \bvht + {b \nabla \cdot \bv_{htt}}, \psi_{ht})_{L^2}\\=(f_t,  \psi_{htt})_{L^2} -2\tilde{k} (\alpha_{ht}\psi_{htt},  \psi_{htt})_{L^2}, \end{multlined}\\[2mm]
		&  \begin{multlined}[t](\bv_{htt}, {c^2 \bvht + b \bv_{htt}})_{L^2}+(\psi_{htt}, \nabla \cdot  ({c^2 \bvht + b \bv_{htt}}))_{L^2}= 0.\end{multlined}\\[2mm]
	\end{aligned} \right.
\end{equation}
This strategy requires the following stronger assumptions on the involved coefficient and source term.
\begin{assumptionp}{W1}\label{Assumption_inv}
	We assume that the coefficient $\alpha_h \in {H^1}(0,T; \spaceS)$ is non-degenerate; that is, there exist $\ulal$, $\olal>0$, independent of $h$, such that
	\[
	0 < 1-2|\tilde{k}|\ulal \leq 1+2\tilde{k}\alpha_h(x,t) \leq 1+2|\tilde{k}|\olal, \quad (x,t) \in \Om \times (0,T).
	\]
Furthermore, there exists $\check\alpha>0$, independent of $h$, such that 
	$$\nLtwoLinf{\alpha_{ht}}\leq \check\alpha.$$
Additionally, we assume that $f \in W^{1,1}(0,T; L^2(\Omega))$.
\end{assumptionp}
Under these assumptions, there exists a unique $(\ph, \bvh) \in W^{3,1}(0,T; \spaceS) \times W^{3,1}(0,T; \spaceV)$ which solves the linearized semi-discrete problem. The stability estimate is then given by
\begin{equation} 
	\begin{aligned}
		&\begin{multlined}[t] 
			E_{\textup{W}}[\psi_h, \bvh](t) := 
			\nLtwo{\psi_{h}(t)}^2 +
			\nLtwo{\psi_{ht}(t)}^2 + \nLtwo{\psi_{htt}(t)}^2 + \nLtwo{\bvh(t)}^2 +  \nLtwo{\bvht(t)}^2 \\ + b \intT \nLtwo{\bv_{ht}(s)}^2 \ds+ b \intT \nLtwo{\bv_{htt}(s)}^2 \ds
			\lesssim\,
				E_{\textup{W}}[\psi_h, \bvh](0) +  \nLoneLtwot{f}^2 + \nLoneLtwot{f_t}^2, \end{multlined}
	\end{aligned}
\end{equation}
where the hidden constant does not depend on $b$. To evaluate the energy at initial time we use \eqref{eq:init_data_mix} to compute $\bvh(0)$ and $\bvht(0)$. Additionally thanks to the smoothness in time of $\psi_{h}$, $\bvh$, and $\alpha_{h}$ we can estimate: 
\[
\nLtwo{\psi_{htt}(0)}^2 \lesssim c^2 \nLtwo{\nabla \cdot \Big(\bvh(0) + \frac{b}{c^2} \bvht(0)\Big) }^2 + \nLtwo{f(0)}^2,
\]
where the hidden constant does not depend on $b$. Observe that for our choice of the approximate initial data \begin{equation}\label{eq:init_eq_proj}
	\begin{aligned}
		(\psi_{0h}, \bv_{0h})=(\tildePp{\psi_0}, \tildePv{\bv_0}), \\
		(\psi_{1h}, \bv_{1h})=(\tildePp{\psi_1}, \tildePv{\bv_1}),
	\end{aligned} 
\end{equation}
we have
\[\nLtwo{\nabla \cdot \Big(\bvh(0) + \frac{b}{c^2} \bvht(0)\Big) } \leq \nLtwo{\nabla \cdot \Big(\bv(0) + \frac{b}{c^2} \bvt(0)\Big) }.\] Therefore, this term is uniformly bounded in $h$ and so is $\nLtwo{\psi_{htt}(0)}$.
Furthermore, note that due to the embedding $W^{1,1}(0,T)\hookrightarrow C[0,T]$, the term $\nLtwo{f(0)}$ can be estimated by $ \nLoneLtwot{f} + \nLoneLtwot{f_t}$.

As will become apparent in the proof of the upcoming Proposition~\ref{Prop:LininvL2}, we need the exact solution $\psi$ to belong to the following space: 
\begin{equation} \label{solution_space2}
	\begin{aligned}
		\tilde X_{r+1} = \left\{ \psi \in \right.&L^{\infty}(0,T; H_0^1(\Omega) \cap H^{r+1}(\Omega)) \,  | \left. \, \pt \in L^{\infty}(0,T; H_0^1(\Omega) \cap H^{r+1}(\Omega)),\right. \\ &\,\left. \ptt \in L^{2}(0,T;  H^{r}(\Omega)),\; \psi_{ttt} \in L^{2}(0,T;  H^{r}(\Omega)) \right\}.
	\end{aligned}
\end{equation}
Compared to the exact solution space \eqref{solution_space} in the non-uniform analysis, here we impose higher regularity on the exact solution (but will also obtain a higher-order-in time error bound). In particular, we additionally require a regularity condition on $\psi_{ttt}$ because we now have to work with the time-differentiated semi-discrete equation. With this, we also need the following alternative assumption on the error of approximation of the variable coefficient.
\begin{assumptionp}{W2}\label{Wes_reg} 
	For a given integer $r \in \{2, \ldots, p^*\}$, we assume that the coefficients $\alpha_h$ its time-derivative $\alpha_{ht}$  approximate $\psi_t$ and $\psi_{tt}$, respectively, up to the following accuracy:
	\begin{equation}
		\begin{aligned}
			&\|\pt-\alpha_h\|_{L^\infty(0,t; L^2)} \leq C_* h^{r} \|\psi\|_{\tilde X_{r+1}},\\
			&{\nLtwoLtwot{\ptt-\alpha_{ht}}} \leq C_* h^{r} \|\psi\|_{\tilde X_{r+1}},
		\end{aligned}
	\end{equation}
for all $t\in [0,T)$, where $C_*>0$ does not depend on $h$ or $b$.
\end{assumptionp}
Unlike before, in the present error analysis we need to establish error bounds not only on the approximate initial data, but also the second time derivative of $\psi_h$ at $t=0$. This is the subject of the following lemma. It is useful to note that $\tilde X^{r+1} \hookrightarrow C^2([0,T]; H^r(\Omega))$.
\begin{lemma}\label{lem:init_cond_ineq}
Let $r \in \{2, \ldots, p^*\}$. With the choice of initial conditions made in \eqref{eq:init_eq_proj}, we have 
	\begin{align}\label{eq:lem_init_1}
		\|\psi_{jh}-\psi_j\|_{L^2} \lesssim h^r \|\psi_j\|_{H^r} \quad \textrm{ for \ $j=0,1$}.
	\end{align}
	Moreover,
	\begin{align}\label{eq:lem_init_2}
	\|\psi_{htt}(0)-\psi_{tt}(0)\|_{L^2} \lesssim h^r \|\psi\|_{\tilde X^{r+1}}.
	\end{align}
\end{lemma}
\begin{proof}
	The bound \eqref{eq:lem_init_1} is a direct result of the properties of the mixed projection; see Lemma~\ref{lem:interpolation_error}. To establish \eqref{eq:lem_init_2}, notice that due to the established regularities, the exact and approximate equations have to hold pointwise in time. In particular, at time zero 
	\begin{equation} 
		\begin{aligned}
			& \begin{multlined}[t]((1+2\tilde{k}\alpha_h(x,0))\big(\psi_{htt}(0)-\psi_{tt}(0)\big), \phi_h)_{L^2} \\- (c^2 \nabla \cdot \big(\bvh(0) - \bv(0)\big) + b \nabla \cdot \big(\bvht(0) - \bvt(0)\big), \phi_h)_{L^2}=2k((\psi_t(0)-\alpha_h(x,0))\psi_{tt}(0), \phi_h),\end{multlined}
		\end{aligned} 
	\end{equation}
	for all $\phi_h \in \spaceS$. Note that due to the choice of initial data and to the definition of the mixed projection, we have
	$$c^2 \nabla \cdot \big(\bvh(0) - \bv(0)\big) + b \nabla \cdot \big(\bvht(0) - \bvt(0)\big) = c^2 \nabla \cdot \big(\tildePv{\bv_0} - \bv_0\big) + b \nabla \cdot \big(\tildePv{\bv_1} - \bv_1\big) = 0.  $$ 
	It then follows, by injecting $\tildePp{\psi_{tt}(0)}$, that 
		\begin{equation} 
		\begin{aligned}
			& \begin{multlined}[t]((1+2\tilde{k}\alpha_h(x,0))\big(\psi_{htt}(0)-\tildePp \psi_{tt}(0)\big), \phi_h)_{L^2} =\\((1+2\tilde{k}\alpha_h(x,0))\big(\psi_{tt}(0)-\tildePp \psi_{tt}(0)\big), \phi_h)_{L^2}+2k((\psi_t(0)-\alpha_h(x,0))\psi_{tt}(0), \phi_h).\end{multlined}\\[2mm]
		\end{aligned} 
	\end{equation}
	Testing with $\phi_h=\big(\psi_{htt}(0)-\tildePp \psi_{tt}(0)\big) \in \spaceS$ yields the following inequality:
	$$ \nLtwo{\psi_{htt}(0)-\tildePp \psi_{tt}(0)} \lesssim \nLtwo{\psi_{tt}(0)-\tildePp \psi_{tt}(0)} +  \nLtwo{\big(\psi_{t}(0)-\alpha_h(x,0)\big)\psi_{tt}(0)}.$$
	We can then use Assumption~\ref{Wes_reg} and the properties of the mixed projection on $\psi_{tt}(0)$ to conclude the proof.
\end{proof}
\noindent Subsequently, we can also state a uniform error estimate for the linearized problem. 

\begin{proposition}\label{Prop:LininvL2}  $b \in [0, \bar{b})$ for some $\bar{b}>0$.  Let $2 \leq r \leq p^*$ and Assumptions~\ref{Assumption_inv} and \ref{Wes_reg} hold. Furthermore, let $\psi \in \tilde X_{r+1}$ be the solution of the exact problem \eqref{eq:continuous_problem} with a sufficiently smooth source term $f$, and coupled with homogeneous Dirichlet data and suitable initial conditions $(\psi(0),\psi_t(0))=(\psi_0,\psi_1)$ and let $v = \nabla \psi$. 	
	Let the approximate initial conditions of \eqref{IBVP_approx_Kuznetsov} be set by \eqref{eq:init_eq_proj}.
	Then the solution $(\ph, \bvh)$ of 
	\eqref{weak_approx_linear} satisfies the following bound:
	\begin{equation} \label{est_lin_inviscid}
		\begin{aligned}
			&\begin{multlined}[t] 
				E_{\textup{W}}[\psi-{\psi}_h, \bv-\bvh](t) \lesssim\, h^{2r} \|\psi\|^2_{\tilde X_{r+1}} \end{multlined}
		\end{aligned}
	\end{equation}
	for $t \in [0,T]$, where the hidden constant is independent of $b$ and $h$, and has the following form: 
	\[
	C_{\textup{lin}} =  C(T) \left(1+C_*^2k^2 \big(\|\psi_{tt}\|_{L^2(0,T; L^\infty)}^2 + \|\psi_{ttt}\|^2_{L^2(0,T; L^\infty)}\big)\right).
	\]
\end{proposition}
\begin{proof}
	Retracing the steps in the proof of Proposition~\ref{Prop:LinL2}, we arrive at
	\begin{equation}
		\begin{aligned}
		E_{\textup{W}}[\tilde{e}_{\psi h}, \tilde{\boldsymbol{e}}_{\boldsymbol{v}h}](t) \lesssim\, E_{\textup{W}}[\tilde{e}_{\psi h}, \tilde{\boldsymbol{e}}_{\boldsymbol{v}h}](0) +  \nLtwoLtwot{\tilde f}^2 + \nLtwoLtwot{\tilde f_t}^2,
		\end{aligned}
	\end{equation}
where we can uniformly bound the initial energy \[E_{\textup{W}}[\tilde{e}_{\psi h}, \tilde{\boldsymbol{e}}_{\boldsymbol{v}h}](0) = \nLtwo{\psi_{htt}(0)-\tildePp \psi_{tt}(0)}\leq h^{2r} \|\psi\|_{\tilde X^r}^2\] thanks to Lemma~\ref{lem:init_cond_ineq}.
\color{black}
	On the other hand, the modified source term is now given by  $$\tilde f = -(1+2\tilde{k}\alpha_h)\tilde{e}_{\psi tt}-2\tilde{k}(\pt-\alpha_h)\psi_{tt},
	$$
	and its time derivative by
	$$\tilde f_t = -(1+2\tilde{k}\alpha_h)\tilde{e}_{\psi ttt}-2\tilde{k}\alpha_{ht}\tilde{e}_{\psi tt} -2\tilde{k}(\pt-\alpha_h)\psi_{ttt} - 2\tilde{k}(\ptt-\alpha_{ht})\psi_{tt}.
	$$
	Then we have using Assumption~\ref{Wes_reg}
	\begin{equation}
		\begin{aligned}
			\|\tilde{f}\|_{L^2(0,t; L^2)} & \leq \begin{multlined}[t] (1+2|\tilde{k}|\olal) \|\tilde{e}_{\psi tt}\|_{L^2(0,t; L^2)}+2|\tilde{k}|\|\psi_t-\alpha_h\|_{L^\infty(0,t; L^2)}\|\psi_{tt}\|_{L^2(0,t; L^\infty)}
			\end{multlined}\\
			& \lesssim \begin{multlined}[t]  h^r \|\psi \|_{\tilde X_{r+1}}+ |\tilde{k}| C_* h^r \|\psi\|_{\tilde X_{r+1}}\|\psi_{tt}\|_{L^2(0,t; L^\infty)}.\end{multlined}
		\end{aligned}
	\end{equation}
	Similarly, 
	\begin{equation}
		\begin{aligned}
			\|\tilde{f_t}\|_{L^2(0,t; L^2)} 
			& \lesssim \begin{multlined}[t]  h^r \|\psi \|_{\tilde X_{r+1}}+ |\tilde{k}| C_* h^r \|\psi\|_{\tilde X_{r+1}}\left({\|\psi_{tt}\|_{L^\infty(0,t; L^\infty)}} + \|\psi_{ttt}\|_{L^2(0,t; L^\infty)}\right), \end{multlined}
		\end{aligned}
	\end{equation}
which yields \eqref{est_lin_inviscid}.
\end{proof}

Given $r \in \{2, \ldots, p^*\}$ and approximate initial data $(\psi_{0h}, \psi_{1h}) \in \spaceS \times \spaceS$, we introduce the ball
\begin{equation} \label{ball_West}
	\begin{aligned}
		\mathcal{B}_{\textup{W}}=\left\{\vphantom{\int_0^t}\right.&(\psi_h^*, \bv^*_h) \in {C^2([0,T]; \spaceS)} \times {H^2(0,T; \spaceV)}:\
		(\psi_h^*, \psi^*_{ht},\bv_h^*, \bv^*_{ht})\vert_{t=0} = (\psi_{0h}, \psi_{1h}, \bv_{0h},\bv_{1h}),\\ &\, \sup_{t \in (0,T)} E_{\textup{W}}[\psi-{\psi}^*_h, \bv-\bvh^*](t) \leq {\tilde C_{\textup{nl}}}^2 h^{2r} \|\psi\|^2_{\tilde X_{r+1}} \left.\vphantom{\int_0^t}\right\},
	\end{aligned}
\end{equation}
and proceed to prove a unique solvability of the nonlinear semi-discrete problem in $\mathcal{B}_{\textup{W}}$.
\begin{theorem} \label{Thm:Wes} $b \in [0, \bar{b})$. Let $2 \leq r \leq p^*$ and 
	let $\psi \in \tilde X_{r+1}$ be the solution of the exact problem \eqref{eq:continuous_problem} with a sufficiently smooth source term $f$, and coupled with homogeneous Dirichlet data and suitable initial conditions $(\psi(0),\psi_t(0))=(\psi_0,\psi_1)$ and let $v = \nabla \psi$. 
	Furthermore, let the approximate initial data $(\psi_{0h}, \psi_{1h})$ be chosen as in Proposition~\ref{Prop:LinL2}. Then there exist \[\overline{h} = \overline{h}(\|\psi\|_{\tilde X_{r+1}})<1 \quad \text{and} \quad M = M(\tilde{k},\sigma, T)>0,\]   
	{such that for $0<h<\overline{h}$ and 
		\begin{equation}
			\sup_{t \in (0,T)}\|\ptt(t)\|^2_{L^\infty} + \sup_{t \in  (0,T)}\|\pt(t)\|^2_{L^\infty}\leq M,
		\end{equation}}
	there is a unique $(\psi_h, \bv_h)\in \mathcal{B}_{\textup{W}}$, which solves {\eqref{weak_approx_linear_inviscid} supplemented by approximate initial conditions \eqref{approx_initcond}}. The constant ${\tilde C_{\textup{nl}}}>0$ in \eqref{ball_West} is independent of $h$ and $b$.
\end{theorem}
\begin{proof}
	The proof is similar to that of Theorem~\ref{Thm:Kuzn} based on the Banach fixed-point theorem. The main difference is that we cannot obtain contractivity of the mapping 
	\[\mathcal{F}: \mathcal{B}_{\textup{W}} \ni (\psi_h^*, \bv^*_h) \mapsto (\psi_h, \bv_h), \]
	 in the $E_{\textup{W}}$ norm as  the right-hand side of the difference equation (analogous to \eqref{eq:source_diff}) 
		\begin{align}
		\tilde{f}=-2\tilde{k} \opsi^*_{ht}\psi^{(2)}_{htt}
	\end{align}
 is not uniformly bounded in $W^{1,1}(0,T; \spaceS)$. \\
	\indent 
We have instead strict contractivity in the lower topology norm $E^{\textup{low}}_{\textup{W}}$ by the lower energy estimate \eqref{eq:Wes_contractivity}. Note that $\mathcal{B}_{\textup{W}}$ is closed with respect to the lower topology $E^{\textup{low}}_{\textup{W}}$. Indeed, take a convergent sequence, $\phi_n \xrightarrow[E^{\textup{low}}_{\textup{W}}]{} \phi$, in $\mathcal{B}_{\textup{W}}$. 
	By virtue of the Banach--Alaoglu theorem, $\mathcal{B}_{\textup{W}}$ (a ball centered around the continuous solution $\psi,\bv$) is weakly-$^*$ compact in $E_{\textup{W}}$. Therefore, $\phi_n$ has a weakly-$^*$ convergent subsequence $\phi^*_n \xrightharpoonup[]{*} \phi^* \in \mathcal{B}_{\textup{W}}$ with respect to $E_{\textup{W}}$. Finally, due to uniqueness of limits, we have $\phi^*=\phi \in \mathcal{B}_{\textup{W}}$, which yields the desired result. 
\end{proof}
\color{black}
Note that the previous uniform analysis does not extend in a straightforward manner to the mixed Kuznetsov equation. The main reason is that when $\sigma \neq 0$ the time-differentiation of the first equation in \eqref{weak_approx_linear} introduces the term $(\boldsymbol{\beta}_h(x,t) \cdot \boldsymbol{v}_{htt}, \phi_h)_{L^2}$. Thus, either {$\|\bv_{htt}(t)\|^2_{L^2}$} or $\int_0^t \|\boldsymbol{v}_{htt}\|^2_{L^2} \ds$ have to be incorporated in the semi-discrete energy of the system and a new testing strategy needs to be devised. {This question is left for future work}.

\begin{remark}[On the polynomial degree]
Theorems~\ref{Thm:Kuzn} and \ref{Thm:Wes} impose a lower bound condition on $p^*$, namely that $p^*\geq2$. 
Note that the lower bound on $p^*$ has two origins.
The first one comes from Lemma~\ref{lem:interpolation_error}. Here one can readily show, by retracing the proof of \cite[Theorem 1.1]{johnson1981error}, that 
\begin{equation}
	\begin{aligned}
		\|\psi-\tildePp{\psi}\|_{L^2} \leq&\, Ch^r  \left \| \psi\right\|_{H^{r+1}}, \quad &&1 \leq r \leq p^*,
	\end{aligned}
\end{equation}
under a higher regularity assumption on $\psi$ but where the minimal polynomial order required is reduced. 
The second source of constraint is the need to control the term $h^{r-d/2}\|\psi\|_{X_{r+1}}$ in the course of the proof of Theorem~\ref{Thm:Kuzn} (see, e.g., \eqref{eq:control_bound_discussion}). This could alternatively  be resolved using the sharper inverse estimate $h^{r-d/2+1}\big(\log\frac1{h}\big)^{3-d}\|\psi\|_{X_{r+1}}$ combined with the $L^q$ error bounds established in Theorem~\ref{prop:Linf}.
\end{remark}
\color{black}
\section{Numerical experiments}\label{sec:numexp}
We next illustrate some of the established convergence rates for the mixed formulations with numerical experiments. To this end, we use the automated finite element software FEniCS \cite{alnaes2015fenics, logg2010dolfin} to implement a mixed-finite-element solver for the damped Kuznetsov equation and verify the predicted convergence rates. 
The codes used for simulating the problems described below are made available at \href{https://github.com/m-meliani/mFEM_Kuznetsov}{https://github.com/m-meliani/mFEM$\_$Kuznetsov}.
\\
\indent Although beyond the scope of this paper, we expect the rigorous analysis of the fully discrete problem to use similar techniques to those presented previously; a linearized problem could be studied inspired by the analysis provided in \cite[Chapter 8]{raviart1983introduction} (for e.g., the Newmark scheme) then combined with a suitable fixed-point approach. 
We also refer the reader to~\cite{kaltenbacher2021convergence} for a study of implicit Runge-Kutta methods for damped nonlinear acoustic wave equations.
\subsection*{An experiment with known exact solution}
We manufacture the solution on a unit square as follows:
\[
\psi(t,x,y) = A \sin(\omega t)\sin(lt)\sin(lt),
\]
for constants $A= 10^{-2}$, $\omega = 6\pi$, and $l=\pi$. Moreover, we fix the model parameters to be $c = 100$, $b= 6\times10^{-9}$, $\sigma = 1$, and $k= 0.5$.

Given a time step $\Delta t$, we integrate $\psi_h$ in time using a predictor-corrector Newmark scheme \cite{kaltenbacher2007numerical}.
We recall that the Newmark scheme is second-order accurate for $\gamma = 0.5$ and unconditionally stable for $\beta = 0.25$; see \cite[Sec. 9.1.1]{hughes2012finite}. We pick these values for our test.
On the other hand, $\bvh$ is integrated through an implicit Euler scheme.

We show here the results for RT$_{1}$ and BDM$_3$ elements, for which the analysis above predicts a convergence rate of $2$ and $3$ respectively. In Figure~\ref{fig:toyconv1}, we plot the error versus discretization step $h$ in a logarithmic scale. The slope of the graph indicates then the order of convergence.

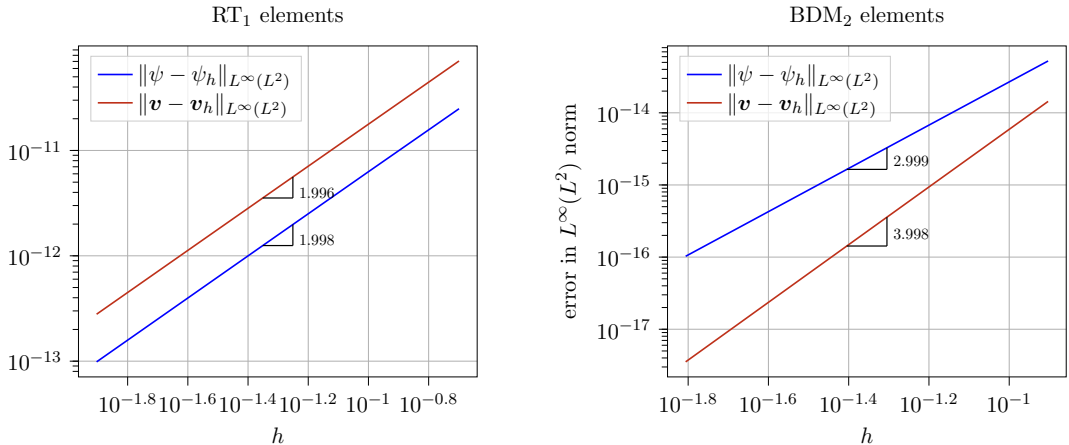
\begin{figure}
	\begin{subfigure}{.5\textwidth}
		\resizebox{!}{6cm}
		{
\begin{tikzpicture}

\begin{axis}[
	title = RT$_1$ elements,
legend cell align={left},
legend style={
  fill opacity=0.8,
  draw opacity=1,
  text opacity=1,
  at={(0.03,0.97)},
  anchor=north west,
  draw=white!80!black
},
log basis x={10},
log basis y={10},
tick align=outside,
tick pos=left,
x grid style={white!69.0196078431373!black},
xlabel={$h$},
xmajorgrids,
xmin=0.0108818820412015, xmax=0.229739670999407,
xmode=log,
xtick style={color=black},
y grid style={white!69.0196078431373!black},
ymajorgrids,
ymin=7.09026771500737e-14, ymax=9.81819917388076e-11,
ymode=log,
ytick style={color=black}
]
\addplot [thick, tumblau]
table {%
0.2 2.48900410409846e-11
0.1 6.28384803002743e-12
0.05 1.57506281776091e-12
0.025 3.93953387593493e-13
0.0125 9.85036090764521e-14
};
\addlegendentry{$\|\psi - \psi_h\|_{L^\infty(L^2)}$}
\addplot [thick, RUred]
table {%
0.2 7.06711777109092e-11
0.1 1.7755130432523e-11
0.05 4.45435752852486e-12
0.025 1.11521937283927e-12
0.0125 2.79073412983974e-13
};
\addlegendentry{$\|\bv - \bvh\|_{L^\infty(L^2)}$}
\addplot [semithick, black, forget plot]
table {%
0.0561009227150982 1.2512158958887e-12
0.0445625469066873 1.2512158958887e-12
};
\addplot [semithick, black, forget plot]
table {%
0.0561009227150982 1.98202818243625e-12
0.0561009227150982 1.2512158958887e-12
};
\addplot [semithick, black, forget plot]
table {%
0.0561009227150982 3.53908163965635e-12
0.0445625469066873 3.53908163965635e-12
};
\addplot [semithick, black, forget plot]
table {%
0.0561009227150982 5.60444236533969e-12
0.0561009227150982 3.53908163965635e-12
};
\draw (axis cs:0.0563814273286737,1.25982733241374e-12) node[
  scale=0.7,
  anchor=base west,
  text=black,
  rotate=0.0
]{1.998};
\draw (axis cs:0.0563814273286737,3.56288234557906e-12) node[
  scale=0.7,
  anchor=base west,
  text=black,
  rotate=0.0
]{1.996};
\end{axis}

\end{tikzpicture}
		}
	\end{subfigure}%
	\begin{subfigure}{.5\textwidth}
		\resizebox{!}{6cm}
		{
\begin{tikzpicture}

\begin{axis}[
	title = BDM$_2$ elements,
legend cell align={left},
legend style={
  fill opacity=0.8,
  draw opacity=1,
  text opacity=1,
  at={(0.03,0.97)},
  anchor=north west,
  draw=white!80!black
},
log basis x={10},
log basis y={10},
tick align=outside,
tick pos=left,
x grid style={white!69.0196078431373!black},
xlabel={$h$},
xmajorgrids,
xmin=0.0140820384782942, xmax=0.138696184008481,
xmode=log,
xtick style={color=black},
y grid style={white!69.0196078431373!black},
ylabel={error in \(\displaystyle L^\infty(L^2)\) norm},
ymajorgrids,
ymin=2.18483087773087e-18, ymax=8.45241474654576e-14,
ymode=log,
ytick style={color=black}
]
\addplot [thick, tumblau]
table {%
0.125 5.22944029897038e-14
0.0625 6.56205436836384e-15
0.03125 8.21048641416584e-16
0.015625 1.02655846065397e-16
};
\addlegendentry{$\|\psi - \psi_h\|_{L^\infty(L^2)}$}
\addplot [thick, RUred]
table {%
0.125 1.43876830627453e-14
0.0625 9.02451188838457e-16
0.03125 5.64726506768229e-17
0.015625 3.53137155677578e-18
};
\addlegendentry{$\|\bv - \bvh\|_{L^\infty(L^2)}$}
\addplot [semithick, black, forget plot]
table {%
0.0495866786033354 1.64351608241705e-15
0.0393880988808277 1.64351608241705e-15
};
\addplot [semithick, black, forget plot]
table {%
0.0495866786033354 3.27819476893913e-15
0.0495866786033354 1.64351608241705e-15
};
\addplot [semithick, black, forget plot]
table {%
0.0495866786033354 1.42468783074602e-16
0.0393880988808277 1.42468783074602e-16
};
\addplot [semithick, black, forget plot]
table {%
0.0495866786033354 3.57719141276525e-16
0.0495866786033354 1.42468783074602e-16
};
\draw (axis cs:0.0498346119963521,1.85692491161868e-15) node[
  scale=0.7,
  anchor=base west,
  text=black,
  rotate=0.0
]{2.999};
\draw (axis cs:0.0498346119963521,1.80601325780575e-16) node[
  scale=0.7,
  anchor=base west,
  text=black,
  rotate=0.0
]{3.998};
\end{axis}

\end{tikzpicture}
		}
	\end{subfigure}%
	\caption{Convergence rates for RT$_{1}$ (left) and BDM$_3$ (right) mixed FEM approximation of a problem with known exact solution}
	\label{fig:toyconv1}
\end{figure}

For RT$_k$ elements, the numerical experiment matches the theoretical findings. Numerical experiments with BDM$_k$ elements hint at a higher order of convergence for the vector variable $\bv$. Although this is not completely surprising in light of the higher accuracy of $\nLtwo{\bv_{ht} - \Ih \bv}$ and of the results established for ``unperturbed problems" in \cite[Theorem 5.2.5]{boffi2013mixed}, this result indicates that the approximation theory can be optimized for the gradient of the ultrasound field in the case of BDM elements.

\subsection*{An experiment with unknown solution}
We next test our results in the case where nonlinear steepening of the wave front is apparent. To this end, we solve the Kuznetsov equation \eqref{IBVP_approx_Kuznetsov} on a 2D square domain, see Figure~\ref{fig:steep1}. The following values are chosen as parameters and source term of the problem
\begin{align}
	\begin{array}{rl}
		b =& 6e-9,\ \ \
		c = 1500,\ \ \
		k = 15,\ \ \
		\sigma = 1,\\
		f =& \frac{400}{\sqrt{3\cdot 10^{-2}}} \exp(-5 \cdot 10^4 t) \exp\left(-\frac{(x-0.5)^2+(y-0.5)^2}{18\cdot 10^{-4}}\right).
	\end{array}
\end{align}
In this way a source term centered around $(x,y) = (0.5,0.5)$ and quickly decaying in time creates a wave that travels spherically away from the center. Note that the value of $k$ is much larger than what would generally be expected of usual propagation media. However, in the present setting the wave energy is quickly diluted due to the spherical propagation of the wave leading to a quick attenuation of the peak. A higher value of $k$ is then chosen so that the wave exhibits visible steepening.

To deal with the steepening of the wave we use the Newmark scheme with parameters $\gamma = 0.85$ and $\beta = 0.45$ for the time integration of $\psi_h$. The velocity
$\bvh$ is, as before, integrated through an implicit Euler scheme.

We show in Figure~\ref{fig:steep1} a comparison between the nonlinear (color gradient) and the damped linear solution (pink) which we calculated by setting $k=\sigma = 0$. We can easily see the steepening appearing at the wave front.
\begin{figure}[H]
	\begin{subfigure}{.5\textwidth}
		\resizebox{!}{5cm}
		{
			\includegraphics[clip,trim = {2.2cm 5cm 2.2cm 5cm}]{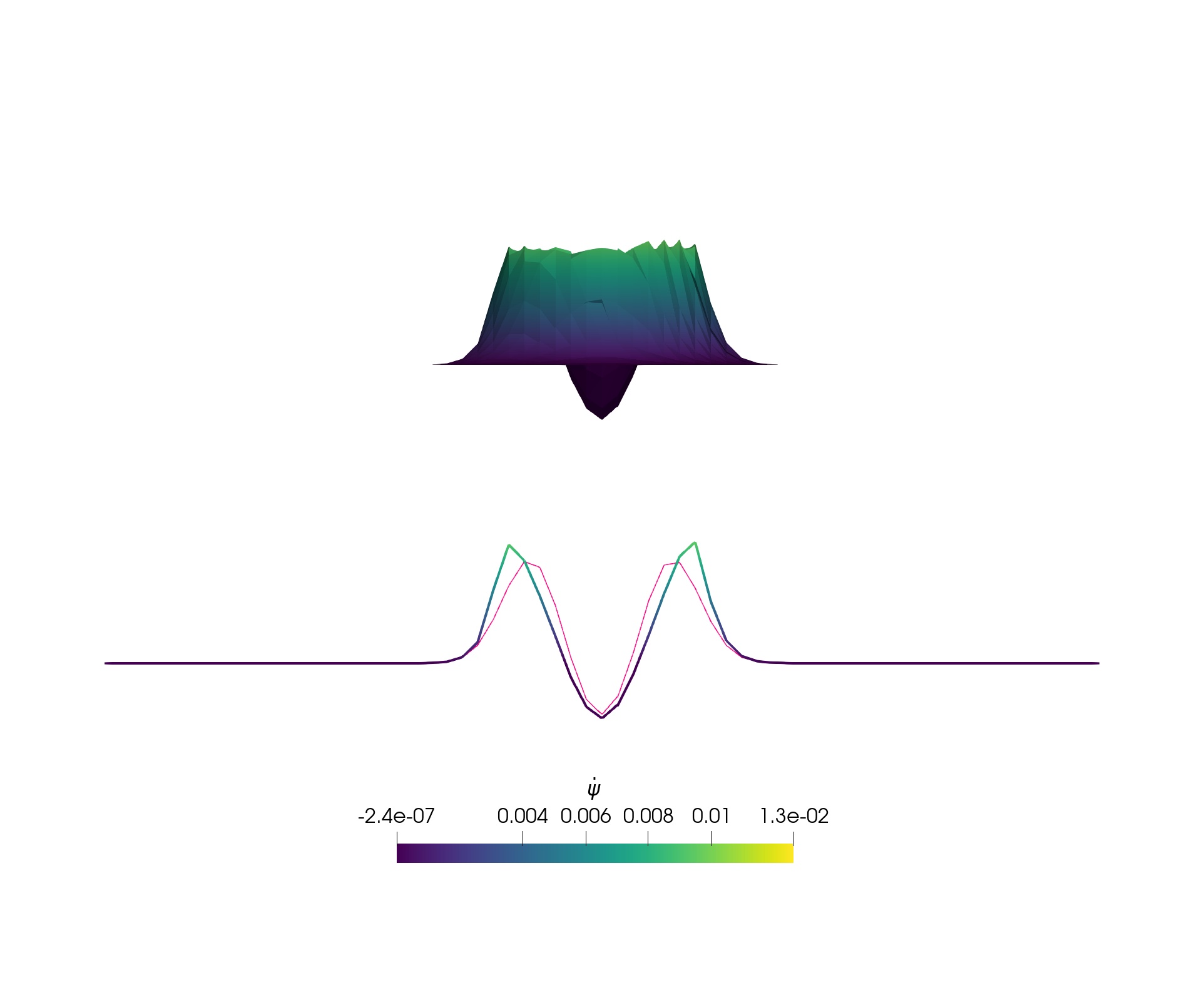}
		}
		\caption{Profile of $\dot{\psi}$ at $t_1$}
	\end{subfigure}%
	\begin{subfigure}{.5\textwidth}
		\resizebox{!}{5cm}
		{
			\includegraphics[clip,trim = {2.2cm 5cm 2.2cm 5cm}]{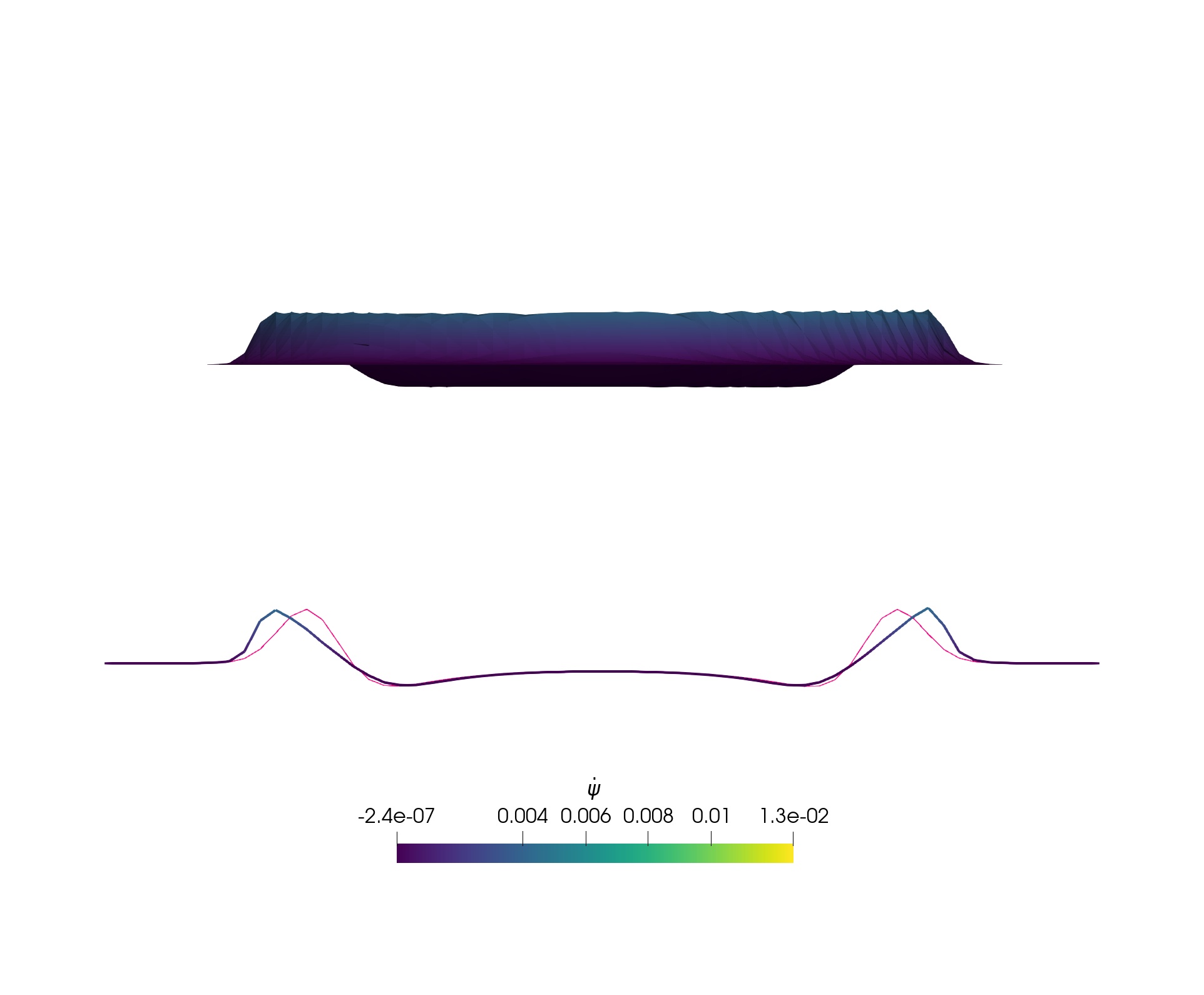}
		}
		\caption{Profile of $\dot\psi$ at $t_2$}
	\end{subfigure}%
	\caption{Comparison of linear (pink) and Kuznetsov (color gradient) solutions along the line $x = 0.5$ at two different time stamps ($t_2> t_1$)}
	\label{fig:steep1}
\end{figure}

In this case we do not have access to the exact solution so we cannot directly compute the norms of the errors $\psi-\psi_h$ and $\bv-\bvh$. 
However we can compute the solution on a finer grid and use that as a reference solution to compute the error. In our case we chose the finer grid to have 128 elements in each spatial direction. We show again the graph of error versus discretization step $h$ in Figure~\ref{fig:errnum2}.

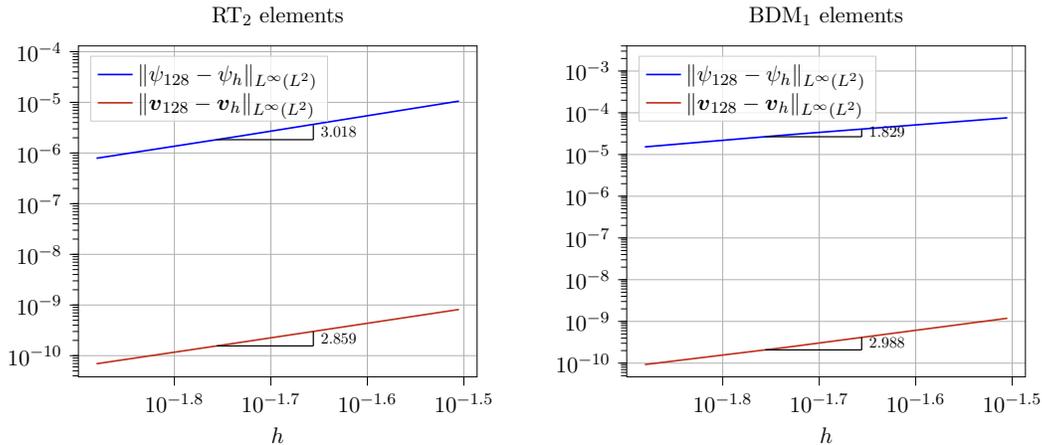
\begin{figure}[H]
	\begin{subfigure}{.5\textwidth}
		\resizebox{!}{6cm}
		{
\begin{tikzpicture}

\begin{axis}[
	title = RT$_2$ elements,
legend cell align={left},
legend style={
  fill opacity=0.8,
  draw opacity=1,
  text opacity=1,
  at={(0.03,0.97)},
  anchor=north west,
  draw=white!80!black
},
log basis x={10},
log basis y={10},
tick align=outside,
tick pos=left,
x grid style={white!69.0196078431373!black},
xlabel={$h$},
xmajorgrids,
xmin=0.0126009482217647, xmax=0.0326312118175446,
xmode=log,
xtick style={color=black},
y grid style={white!69.0196078431373!black},
ymajorgrids,
ymin=3.81095178130659e-11, ymax=13e-05,
ymode=log,
ytick style={color=black}
]
\addplot [thick, tumblau]
table {%
	0.03125 1.05799917728602e-05
	0.0185185185185185 2.15440544270739e-06
	0.0131578947368421 7.88042291243316e-07
};
\addlegendentry{$\|\psi_{128} - \psi_h\|_{L^\infty(L^2)}$}
\addplot [thick, RUred]
table {%
	0.03125 8.15699538539413e-10
	0.0185185185185185 1.81747414414312e-10
	0.0131578947368421 6.92222079899262e-11
};
\addlegendentry{$\|\bv_{128} - \bvh\|_{L^\infty(L^2)}$}
\addplot [semithick, black, forget plot]
table {%
	0.0220739924594722 1.83445555046674e-06
	0.0175339954636496 1.83445555046674e-06
};
\addplot [semithick, black, forget plot]
table {%
	0.0220739924594722 3.67549146676572e-06
	0.0220739924594722 1.83445555046674e-06
};
\addplot [semithick, black, forget plot]
table {%
	0.0220739924594722 1.55759388889163e-10
	0.0175339954636496 1.55759388889163e-10
};
\addplot [semithick, black, forget plot]
table {%
	0.0220739924594722 3.00838649681604e-10
	0.0220739924594722 1.55759388889163e-10
};
\draw (axis cs:0.0221843624217695,2.07730990995974e-06) node[
scale=0.7,
anchor=base west,
text=black,
rotate=0.0
]{3.018};
\draw (axis cs:0.0221843624217695,1.73174490922695e-10) node[
scale=0.7,
anchor=base west,
text=black,
rotate=0.0
]{2.859};
\end{axis}

\end{tikzpicture}
		}
	\end{subfigure}%
	\begin{subfigure}{.5\textwidth}
		\resizebox{!}{6cm}
		{
\begin{tikzpicture}

\begin{axis}[
	title = BDM$_1$ elements,
legend cell align={left},
legend style={
  fill opacity=0.8,
  draw opacity=1,
  text opacity=1,
  at={(0.03,0.97)},
  anchor=north west,
  draw=white!80!black
},
log basis x={10},
log basis y={10},
tick align=outside,
tick pos=left,
x grid style={white!69.0196078431373!black},
xlabel={$h$},
xmajorgrids,
xmin=0.0126009482217647, xmax=0.0326312118175446,
xmode=log,
xtick style={color=black},
y grid style={white!69.0196078431373!black},
ymajorgrids,
ymin=4.68812708434548e-11, ymax=4e-3,
ymode=log,
ytick style={color=black}
]
\addplot [thick, tumblau]
table {%
	0.03125 7.56868725886078e-05
	0.0185185185185185 2.96552836549757e-05
	0.0131578947368421 1.52279214303462e-05
};
\addlegendentry{$\|\psi_{128} - \psi_h\|_{L^\infty(L^2)}$}
\addplot [thick, RUred]
table {%
	0.03125 1.19044096241639e-09
	0.0185185185185185 2.42511755824251e-10
	0.0131578947368421 9.26018007867195e-11
};
\addlegendentry{$\|\bv_{128} - \bvh\|_{L^\infty(L^2)}$}
\addplot [semithick, black, forget plot]
table {%
	0.0220739924594722 2.66581993034641e-05
	0.0175339954636496 2.66581993034641e-05
};
\addplot [semithick, black, forget plot]
table {%
	0.0220739924594722 4.0614841334614e-05
	0.0220739924594722 2.66581993034641e-05
};
\addplot [semithick, black, forget plot]
table {%
	0.0220739924594722 2.07919937048201e-10
	0.0175339954636496 2.07919937048201e-10
};
\addplot [semithick, black, forget plot]
table {%
	0.0220739924594722 4.1367512939778e-10
	0.0220739924594722 2.07919937048201e-10
};
\draw (axis cs:0.0221843624217695,2.63237509178516e-05) node[
scale=0.7,
anchor=base west,
text=black,
rotate=0.0
]{1.829};
\draw (axis cs:0.0221843624217695,2.34621474703803e-10) node[
scale=0.7,
anchor=base west,
text=black,
rotate=0.0
]{2.988};
\end{axis}

\end{tikzpicture}
		}
	\end{subfigure}%
	\caption{Convergence rates for RT$_{2}$ and BDM$_2$-element approximations of a problem with unknown exact solution}
	\label{fig:errnum2}
\end{figure}

Similarly to the problem with known solution, the numerical experiments for RT$_k$ element match the theoretical findings. Numerical experiments with BDM$_k$ elements again hint at a higher order of convergence for the vector variable $\bv$ than proven.

\section{Conclusion} 
In this work, we have performed the well-posedness and \emph{a priori} error analysis of the semi-discrete Kuznetsov equation in mixed form, which allows us to characterize the full acoustic field accurately at once. In particular, we have studied both potential-velocity and pressure-velocity forms and established convergence estimates for a broad family of mixed finite elements, including the popular RT and BDM elements. {Additionally, we have provided a uniform-in-$b$ error analysis of the mixed approximation of the Westervelt equation $(\sigma=0)$ and established sufficient conditions under which  the hidden constant in the derived bounds does not degenerate as $b \rightarrow 0^+$.} We have demonstrated our findings through numerical experiments which reinforced our results and even out-performed them in the case of BDM approximations of the acoustic particle velocity. 

\section*{Acknowledgments}
The authors would like to thank Prof.\ A.K.\ Pani for his helpful comments regarding the maximum-norm error estimates.
\bibliography{references}
\bibliographystyle{siam} 
\end{document}